\documentclass[3p,times]{elsarticle}
\biboptions{sort&compress}

\usepackage[nice]{nicefrac}
\newcommand{\onehalf}{\nicefrac{1}{2}}

\usepackage{amssymb}
\usepackage{amsthm}
\usepackage{amsmath}
\usepackage{graphicx}
\usepackage[colorlinks=true, allcolors=blue]{hyperref}
\usepackage{xcolor}
\usepackage{siunitx}
\usepackage{booktabs}
\usepackage{url}
\usepackage{enumitem}
\usepackage{stmaryrd}
\usepackage{mathsemantics}  

\newcommand{\tn}{|\mspace{-1mu}|\mspace{-1mu}|}
\newcommand{\normv}[1]{\|#1\|_{\bV_h}}
\newcommand{\normvast}[1]{\|#1\|_{\bV_h,*}}
\newcommand{\normt}[1]{\|#1\|_{Q_{T,h}}}
\newcommand{\normtast}[1]{\|#1\|_{Q_{T,h},*}}
\newcommand{\normf}[1]{\|#1\|_{Q_{F,h}}}
\newcommand{\normfast}[1]{\|#1\|_{Q_{F,h},*}}

\newcommand{\tnormv}[1]{\tn#1\tn_{\bV_h}}
\newcommand{\tnormt}[1]{\tn#1\tn_{Q_{T,h}}}
\newcommand{\tnormf}[1]{\tn#1\tn_{Q_{F,h}}}

\newtheorem{theorem}{Theorem}[section]
\newtheorem{remark}{Remark}[section]
\newtheorem{lemma}{Lemma}[section]
\newtheorem{corollary}{Corollary}[section]
\newtheorem{assumption}{Assumption}[section]

\numberwithin{equation}{section}
\numberwithin{theorem}{section}
\numberwithin{table}{section}
\numberwithin{figure}{section}

\allowdisplaybreaks

\begin{document}

\begin{frontmatter}



\title{A cut finite element method for the Biot system of poroelasticity}



\author[1]{Nanna Berre\corref{cor}}
\ead{nanna.berre@ntnu.no}
\address[1]{Department of Mathematical Sciences, Norwegian University of Science and Technology, Trondheim, Norway}

\author[2,3]{Kent-Andre Mardal}
\address[2]{Department of Numerical Analysis and Scientific Computing, Simula Research Laboratory, Oslo,
Norway}
\address[3]{Department of Mathematics, University of Oslo, Oslo, Norway}
\ead{kent-and@simula.no}

\author[1]{Andr{\'e} Massing}
\ead{andre.massing@ntnu.no}

\author[4]{Ivan Yotov}
\address[4]{Department of Mathematics, University of Pittsburgh, Pittsburgh, USA}
\ead{yotov@math.pitt.edu}

\cortext[cor]{Corresponding author}

\begin{abstract}
We propose a novel cut finite element method for the numerical solution of the Biot system of poroelasticity.
The Biot system couples elastic deformation of a porous solid with viscous fluid flow and commonly arises on domains with complex geometries that make high-quality volumetric meshing challenging. To address this issue, we employ the cut finite element framework, where the domain boundary is represented independently of the background mesh, 
which significantly simplifies the meshing process.
Our approach builds upon a parameter robust total pressure formulation of the Biot system, which we combine with the cut finite element method to develop a geometrically robust solution scheme while preserving parameter robustness. 
A key ingredient in the theoretical analysis is a modified inf-sup condition which also holds for mixed boundary conditions, leading to stability and optimal error estimates for the proposed formulation.
Finally, we provide numerical evidence demonstrating the theoretical properties of the method and showcasing its capabilities by solving the Biot system on a realistic brain geometry. 
\end{abstract}




\end{frontmatter}



\section{Introduction}
The Biot system~\cite{biot1941general} of poroelasticity models the flow of a viscous fluid in deformable porous
media. It couples a momentum balance equation for the solid structure with a mass conservation equation for the fluid. The fluid pressure affects the solid stress, while the divergence of the solid displacement influences the fluid content, resulting in a fully coupled system.
The numerical solution of the Biot system has been an active area of research due to the numerous applications of the model in diverse fields, such as biomedicine, blood flow and organ tissue modeling, 
geosciences, groundwater flow, hydraulic fracturing, carbon sequestration, and landslides~\cite{Stoverud2014,stoverud2016poro,MalandrinoMoeendarbary2019,BukacYotovZakerzadehEtAl2015a,StoverudDarcisHelmigEtAl2012,GambolatiTeatiniBauEtAl2000,SalimzadehUsuiPalusznyEtAl2017,Wang2000,Cheng2016}.
For the finite element-based numerical solution of the Biot system, we refer the reader to two-field displacement--pressure formulations
in \cite{murad1992improved,Zik-MINI}, three-field
displacement--pressure--Darcy velocity formulations in 
\cite{phillips2007coupling1,Zik-stab,Yi-Biot-locking,Zik-nonconf,Lee-Biot-three-field,Yi-Biot-nonconf,phillips-DG,Whe-Xue-Yot-Biot}, three-field displacement--pressure--total pressure formulations in 
\cite{Lee-Mardal-Winther,oyarzua2016locking}, four-field fully mixed stress--displacement--pressure--Darcy velocity formulations in \cite{Yi-Biot-mixed}, and five-field weakly symmetric stress--displacement--rotation--pressure--Darcy velocity formulations in \cite{Lee-Biot-five-field,msfmfe-Biot}. 

The above-mentioned methods are based on finite element partitions
that conform to the domain boundary. This may present challenges in
applications with highly irregular geometries, e.g., brain or fracture
networks, or time-dependent geometries, e.g., blood vessels,
propagating fractures, or sedimentary basins. Unfitted discretization
methods such as the extended finite element method (XFEM)~\cite{MoesDolbowBelytschko1999}
are designed to remedy this issue by embedding the complex or
changing geometry domain into a static domain with simple geometry
partitioned by a regular grid. We refer to
~\cite{FriesBelytschko2000,ChessaBelytschko2003,HansboHansbo2002,Fumagalli2012,FumagalliScotti2014} for early work in this direction.
A major challenge in unfitted methods is to ensure their geometrical robustness so that
stability and accuracy are not compromised by small or arbitrarily cut elements.
As a remedy, the cut finite element method (CutFEM) \cite{Burman-etal-cutfem-IJNME,dePrenterVerhooselvanBrummelenEtAl2023} has emerged 
over the past 15 years 
as a powerful and general finite element methodology that provides a unified stabilization framework for the analysis of unfitted methods.
Similar to the XFEM approach, boundary and interface conditions are weakly
imposed using Nitsche's method
\cite{Nitsche1971,Burman-Hansbo-cutfem-Nitsche}
or the Lagrange multiplier method \cite{Burman-Hansbo-cutfem-LM}.
In addition, a so-called ghost penalty stabilization is 
incorporated into the discrete formulation which makes the resulting
CutFEM scheme amenable to a geometrically robust stability and error analysis.
Alternatively, strongly incorporated stabilization can be employed
through the redefinition of the finite element spaces
using aggregation techniques~\cite{JohanssonLarson2013,BadiaVerdugoMartin2018}. 
As a result, a wide range of problem classes
has been treated 
including, e.g.,
elliptic interface problems~\cite{BurmanZunino2012,
GuzmanSanchezSarkis2017,GuerkanMassing2019},
Stokes and Navier-Stokes type problems
\cite{BurmanHansbo2013,MassingLarsonLoggEtAl2013,BurmanClausMassing2015,
CattaneoFormaggiaIoriEtAl2014,MassingSchottWall2017,WinterSchottMassingEtAl2017,BadiaMartinVerdugo2018,GuzmanOlshanskii2016,DokkenJohanssonMassingEtAl2020,BurmanFreiMassing2022},
and complex multiphysics interface problems
~\cite{SchottRasthoferGravemeierEtAl2015,ClausKerfriden2019,FrachonZahedi2019,MassingLarsonLoggEtAl2015,BerreRognesMassing2024}.
As a natural application area,
unfitted finite element methods have also been proposed
for problems in porous media~\cite{FrachonHansboNilssonEtAl2024,LehrenfeldvanBeeckVoulis2023,HansboLarsonMassing2017},
possibly including discrete fracture networks~\cite{FlemischFumagalliScotti2016,Fumagalli2012,DAngeloScotti2012,FormaggiaFumagalliScottiEtAl2013,DelPraFumagalliScotti2017,BurmanHansboLarsonEtAl2018a}.
We refer to \cite{BordasBurmanLarsonEtAl2018,dePrenterVerhooselvanBrummelenEtAl2023}
and references therein for a comprehensive list of
applications.
An alternative approach is the shifted boundary
method \cite{SBM}, which applies modified boundary conditions on a
surrogate boundary consisting of faces of elements in the background
mesh. 

In this paper, we employ the CutFEM methodology for the solution of the
Biot system of poroelasticity, an approach that has so far received
relatively little attention in the literature. CutFEM for the Biot
consolidation model in a two-field displacement--pressure formulation
has been studied in \cite{CerroniRaduZuninoEtAl2019,Liu-etal}. In
\cite{CerroniRaduZuninoEtAl2019}, the method is applied for modeling
sedimentary basins with moving domains. The focus is on developing a
fixed-stress iterative coupling method and the analysis of the
convergence of the iteration. In \cite{Liu-etal}, CutFEM is applied to
the coupled Biot system. Extensive numerical studies are presented,
but the theoretical stability and convergence of the method are not
addressed. A CutFEM-based approach for coupling incompressible
fluid flow with poroelasticity is developed in \cite{Ager-etal}, where
numerical analysis is not presented.

\subsection{Novel contributions and outline of the paper}
The goal of this paper is to develop and analyse an unfitted
CutFEM-based formulation for the Biot system of poroelasticity using a
three-field displacement--pressure--total pressure formulation
\cite{Lee-Mardal-Winther,oyarzua2016locking}. The total pressure
formulation has gained significant popularity due to its robustness
with respect to the physical parameters. In particular, it avoids
locking when the Lam\'e parameter $\lambda \to \infty$ and pressure
oscillations when the hydraulic conductivity $K \to 0$. Our method is designed
to carry over this parameter robustness to the cut finite element
setting, while ensuring robustness with respect to the size of the cut
elements. The latter is achieved by utilizing suitably scaled ghost penalty
stabilization. 
We allow for both essential and natural boundary conditions for 
the solid and the
fluid, with the essential boundary conditions being weakly
enforced using Nitsche's method.

We develop stability and error analysis of the proposed method. The
norm of the fluid pressure has terms that are scaled by either $K$ or
$\lambda^{-1}$ and we keep track of the dependence of the bounds on
these parameters. As the finite elements for the displacement and
total pressure form a Stokes-stable pair, we employ and extend previous results
on CutFEM for the Stokes equations \cite{BurmanHansbo2013,
Massing-etal-cutfem-Stokes, GuzmanOlshanskii2016}. A key component in
the stability analysis is establishing an inf-sup condition over the
entire physical domain. 
The analysis in~\cite{GuzmanOlshanskii2016} for the Stokes equations
assumes Dirichlet boundary conditions for the velocity everywhere,
and thus a zero mean value of the pressure.
However, zero mean value for the total pressure
cannot be assumed in the presented Biot formulation,
as we need to consider both a perturbed saddle point problem as well as
mixed boundary conditions.
To deal with this issue, we combine the approach in
\cite{GuzmanOlshanskii2016} with an argument from
\cite{Massing-etal-cutfem-Stokes,MassingLarsonLoggEtAl2013}, 
to obtain a modified
displacement--total pressure inf-sup condition containing a pressure ghost penalty as
defect term.
This approach necessitates the construction of a special interpolant with a
certain orthogonality property on the physical boundary, which is done by a
patch-wise argument, extending the construction in
\cite{BeckerBurmanHansbo2009}. As a result, we are able to
establish stability and optimal-order error estimates for the method
with constants independent of the size of the cut elements. We present
extensive numerical results in two and three dimensions verifying the
stability and convergence properties of the method and its robustness
with respect to the physical parameters and how the physical domain 
cuts the background mesh.
Furthermore, we present an example using the method to model the flow in and
deformation of the brain, which illustrates its robustness for complex physical
applications with highly irregular three dimensional geometry and
parameters in the locking regime. 

The paper is structured as follows. In Section~\ref{sec:model-problem} we
introduce the Biot system and briefly motivate the total pressure formulation, 
followed by the presentation of the proposed CutFEM formulation in Section~\ref{sec:cutfem-biot}. 
Afterwards, we collect necessary inequalities, interpolation results and
present a detailed construction of the modified interpolant together with a brief
discussion of ghost penalty realizations in
Section~\ref{sec:preliminaries}. 
The main theoretical results, including stability, inf-sup conditions, and the final
a priori error estimate are derived in Section~\ref{sec:analysis}.
Finally, extensive numerical results are given in
Section~\ref{sec:numerics}, before we draw our conclusions in
Section~\ref{sec:conclusion}.

\section{Model problem}
\label{sec:model-problem}
Let $\Omega \subset \bbR^n$, $n = 2,3$, be an open bounded domain with boundary \(\Gamma\) of class \(C^2\). 
The quasi-static two field Biot's equations for a poroelastic material read 
\begin{align*}
-\mu \nabla \cdot \varepsilon (\bu) - \lambda \nabla \nabla \cdot \bu + \alpha \nabla p_F  &= \bf, \\
\alpha \frac{\partial}{\partial t} \nabla \cdot \bu + c_0 \frac{\partial}{\partial t} p_F - K \Delta p_F & = g,
\end{align*}
with the fundamental unknowns being the solid displacement $\bu$ and
the fluid pressure $p_F$. A main challenge with Biot's equations is
the large number of parameters (here assumed constant throughout the
domain) $\mu$, $\lambda$, $\alpha$, $c_0$, and $K$, denoting the two
Lam\'e parameters (for simplicity \(\mu\) scaled by 2), the Biot-Willis constant, the constrained specific
storage coefficient, and the hydraulic conductivity, respectively.  In
particular, locking may appear if $\lambda \gg \mu$ and pressure
oscillations may appear as $K\rightarrow 0$. 

To simplify the notation in the following, we focus only on the parameters
that cause numerical challenges. That is, let us  consider
$\alpha=1$, $c_0=1/\lambda$ and a backward Euler time-discretization with $\Delta t=1$. Then the above system reduces to
\begin{align*}
-\mu \nabla \cdot \varepsilon (\bu) - \lambda \nabla \nabla \cdot \bu + \nabla p_F  &= \bf, \\
-\nabla \cdot \bu - \frac{1}{\lambda} p_F + K \Delta p_F & = g.
\end{align*}
Here, the sign is flipped on the second equation to obtain symmetry,
and $g$ is altered to accommodate the sign change and source terms
resulting from the time-discretization.

To briefly motivate the total pressure formulation, we introduce the solid pressure $p_S = -\lambda \nabla \cdot \bu$, arriving at
\begin{align*}
-\mu \nabla \cdot \varepsilon (\bu) + \nabla p_S + \nabla p_F &= \bf,
\\
-\nabla \cdot \bu - \frac{1}{\lambda} p_S & = 0,
\\
-\nabla \cdot \bu   - \frac{1}{\lambda} p_F + K \Delta p_F & = g. 
\end{align*}
In the case $K\rightarrow 0$ and $\lambda\rightarrow\infty$ we have two pressures to be determined by $\nabla \cdot \bu$ and as such the system is degenerate.  
In the total pressure formulation an alternative pressure is introduced, namely $p_T = p_F + p_S$. 
The system then becomes 
\begin{align}
-\mu \nabla \cdot \varepsilon (\bu)  + \nabla p_T   &= \bf, 
\label{eq:biot-total-press-I}
\\
-\nabla \cdot \bu - \frac{1}{\lambda} p_T + \frac{1}{\lambda} p_F  &= 0, 
\label{eq:biot-total-press-II}
\\
 \frac{1}{\lambda} p_T  -  \frac{2}{\lambda} p_F + K \Delta p_F & = g. 
\label{eq:biot-total-press-III}
\end{align}
Here, the degeneracy is eliminated at least on the continuous level, 
although, numerical issues may be present for certain schemes
as $K\rightarrow 0$ and/or $\lambda \rightarrow \infty$. 
We remark that the total pressure formulation \cite{oyarzua2016locking, Lee-Mardal-Winther} 
is just 
one of many Biot formulations, with the advantage that regular inf-sup stable Stokes elements can be used for the displacement and the total pressure.
We also note that as for Stokes problem,  the case of pure Dirichlet  conditions results in an undetermined constant for the pressure if $\lambda=\infty$. For $\lambda$ large but not infinity
the consequence is one bad mode. Such considerations are dealt with in detail in \cite{Lee-Mardal-Winther}, and we will not focus on this case in the following. 

The system \eqref{eq:biot-total-press-I}--\eqref{eq:biot-total-press-III} 
is supplemented with appropriate boundary conditions, which, in accordance with \cite{Lee-Mardal-Winther}, we take to be
\begin{alignat}{3}
\bu &= \bu_D  & &\quad \mbox{ on } \Gamma_d, 
\label{eq:bc-u-diri-inhom}
\\
K \frac{\partial p_F}{\partial n} & = g_N & & \quad \mbox{ on } \Gamma_d,
\label{eq:bc-p_F-neum-inhom}
\\
(\mu \varepsilon (\bu) - p_T \bI) \cdot \bn &= \bsigma_N & & \quad \mbox{ on } \Gamma_s, 
\label{eq:bc-total-stress-inhom}
\\
p_F & = p_{F,D}  & &\quad \mbox{ on } \Gamma_s,
\label{eq:bc-p_F-diri-inhom}
\end{alignat}
where $\Gamma = \Gamma_d \cup \Gamma_s$ is a non-overlapping partition of the boundary.
Throughout the remaining paper, we will make use of the following
notation. The expression $a \lesssim(\gtrsim) \ b$ means that there
exists a constant $C > 0$ independent of $K$, $\lambda$, and the mesh size $h$ such
that $a \leqslant(\ge) \ C \,b$. We write $a \sim b$ if $a \lesssim b
\lesssim a$. Moreover, boldface symbols are used to denote vector
valued functions or spaces. 
For inner products and norms on the Lebesgue space \(L^2(\Omega)\) we use the notation \((\cdot,\cdot)_\Omega\) and \(\norm{\cdot}_\Omega\), respectively. We denote by \( H^m(\Omega)\) the Sobolev space with weak derivatives up to order \(m\) in \(L^2(\Omega)\) equipped with norm \(\norm{\cdot}_{m,\Omega}\). In addition, we use the following notation for an \(H^m(\Omega)\) space with vanishing trace on a part of the boundary \(\Gamma_X\subset \partial \Omega\): 
\begin{align*}
   H^m_{\Gamma_X,0}(\Omega) = \{ p \in H^m(\Omega) \mid p|_{\Gamma_X} = 0 \}.
\end{align*}
We define the trace space  \(H^{\onehalf}(\partial \Omega) = H^1(\Omega)|_{\partial \Omega}\), with norm \(\|v\|_{H^{\onehalf}(\partial \Omega)} = \inf_{w\in H^1(\Omega), w|_{\partial \Omega} = v}   \|w\|_{H^{1}( \Omega)} \), and correspondingly the trace space on part of the boundary, \(H^{\onehalf}(\Gamma_X) = \{ v|_{\Gamma_X} : v \in H^{\onehalf}(\partial \Omega)\}\), with norm \(\|v\|_{H^{\onehalf}(\Gamma_X)} = \inf_{w\in H^1(\Omega), w|_{\Gamma_X} = v}   \|w\|_{H^{1}( \Omega)} \), and \(\widetilde{H}^{-\onehalf}(\Gamma_X) = [H^{\onehalf}(\Gamma_X)]^* \) as its dual space. We assume that \(\bu_D \in [H^{\onehalf}(\Gamma_d)]^n\), \(p_{F,D} \in H^{\onehalf}(\Gamma_s)\), \(\bsigma_N \in [\widetilde{H}^{-\onehalf}(\Gamma_s)]^n\), and \(g_N \in \widetilde{H}^{-\onehalf}(\Gamma_d)\). For simplicity, we denote the dual pairing between \(\widetilde{H}^{-\onehalf}(\Gamma_X)\) and \(H^{\onehalf}(\Gamma_X)\) by \(( \cdot, \cdot )_{\Gamma_X}\).
Set operations involving \(\cP_h\) should be understood as element-wise operations, e.g., \(\cP_h\cap U = \{P \cap U  \mid P \in \cP_h \}\), such that for inner products we have \((\cdot,\cdot)_{\cP_h\cap U} = \sum_{P\in \cP_h} (\cdot,\cdot)_{P\cap U}\).

\section{A cut finite element method for Biot's consolidation model}
\label{sec:cutfem-biot}
We let $\widetilde{\cT}_h$ be a quasi uniform background mesh, consisting of shape regular closed simplices or rectangles/cuboids $T$ covering $\overline{\Omega}$, see Figure \ref{fig:cutfem-set-explanations}. 
The \emph{active mesh} $\cT_h$ is then defined
as the collection of elements which intersect $\Omega$,
\begin{gather*}
\label{eq:active_mesh}
\cT_{h} = \{ T \in \widetilde{\cT }_{h}  \mid  T \cap \Omega  \neq \emptyset  \},
\intertext{which we further decompose into}
\cT_h^{\Gamma} = \{ T \in \cT_{h}  \mid  T \cap \Gamma  \neq \emptyset  \},
\quad
\cT_h^i =  \cT_h \setminus \cT_h^{\Gamma}.
\label{eq:interior_bundary_mesh}
\end{gather*}
The set of interior facets associated with $\cT_h$ is denoted by
\begin{equation*}
    \cF_{h}^{i}
    =
    \{ F = T^{+} \cap T^{-}  \mid  T^{+}, T^{-} \in \cT_{h}
    \text{ and } 0 < |F|_{n-1} < \infty \}.
\end{equation*}
Here, $|F|_{n-1}$ is the $n-1$-dimensional Hausdorff measure of the facet $F$. In the following we will use $\cT_h$, $\cT_h^i$, and $\cT_h^\Gamma$ to also denote the domains consisting of the union of corresponding mesh elements. 

\begin{figure}[t]
    \centering
    \includegraphics[trim = {0cm 0cm 0cm 0cm}, clip, width=0.4\textwidth]{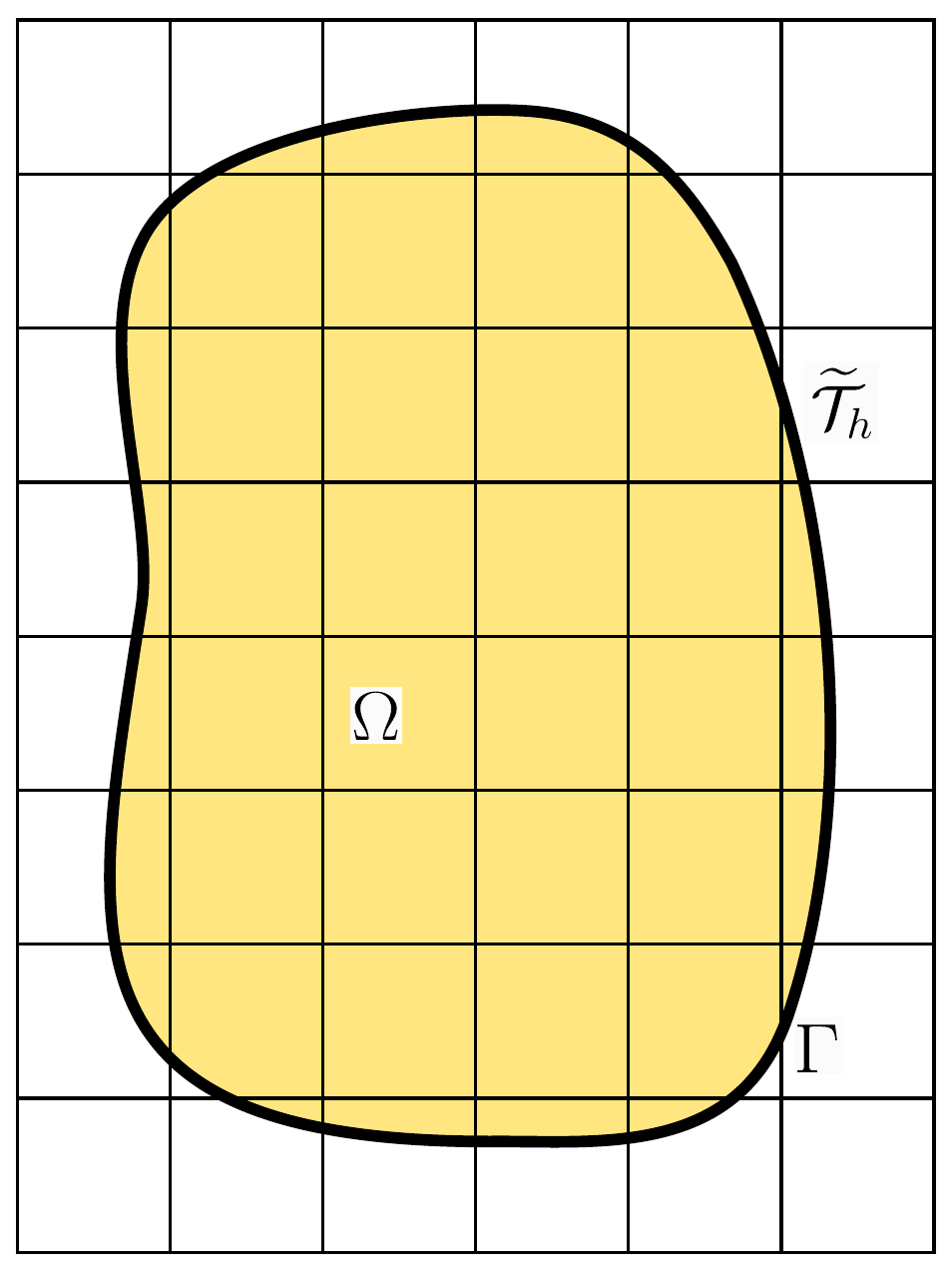}
    \includegraphics[trim = {0cm 0cm 0cm 0cm}, clip, width=0.4065\textwidth]{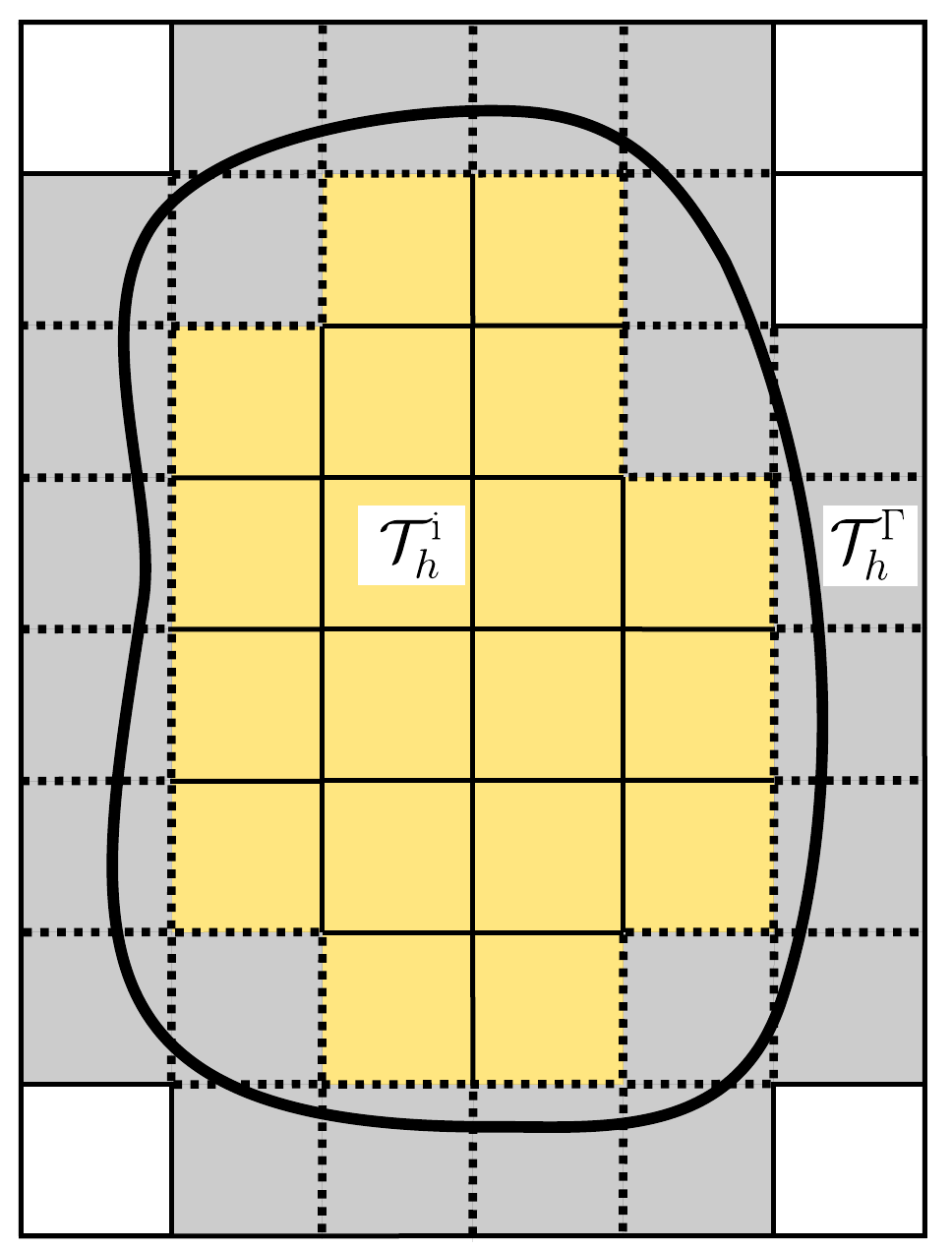}
    \caption{Example of an unfitted mesh, with the domain $\Omega$ inside the boundary $\Gamma$.} 
    \label{fig:cutfem-set-explanations}
\end{figure}

We consider stable Stokes elements to discretize displacement and total pressure. For simplicity we focus on the (possibly higher order) Taylor-Hood elements on simplices. Alongside our presentation we will also briefly comment on possible extensions of
the proposed cut finite element method to other inf-sup stable elements in the spirit of
\cite{GuzmanOlshanskii2016} and the use of strong stabilizations~\cite{BadiaVerdugoMartin2018}.
On the active mesh, we define the discrete base function space to be
\begin{equation*}
    X_h^k \coloneqq \bbP_c^k(\cT_h) = \bbP^k(\cT_h) \cap C^0(\cT_h), \quad\text{where}\quad \bbP^k(\cT_h) \coloneqq \bigoplus_{T\in T_h} \bbP^k(T),
\end{equation*}
where $\bbP^k(T)$ denotes the set of all polynomials up to order $k$ defined on $T$. 
Then, for $k\geqslant 2$, \( l\geqslant 1\),  the discrete function spaces for the displacement, total pressure, and fluid pressure are defined, respectively, as
\begin{align*}
    \bV_h = [X_h^k]^n, \quad Q_{T,h} = X_h^{k-1}, \quad Q_{F,h} = X_h^{l}.
\end{align*}
Note that the function spaces are now defined on the entire active mesh, instead of just the physical domain.

To derive a cut finite element method for Biot's equations, we first need to derive a preliminary 
formulation where all the relevant boundary conditions are enforced weakly.
In this case we do not need to distinguish between the discrete trial and test function spaces.
Note that~\eqref{eq:bc-total-stress-inhom} and ~\eqref{eq:bc-p_F-neum-inhom} are Neumann conditions
for the displacement and fluid pressure, respectively, and thus
will be naturally enforced weakly in the discrete formulation of ~\eqref{eq:biot-total-press-I} and ~\eqref{eq:biot-total-press-III}, respectively.
The Dirichlet conditions \eqref{eq:bc-u-diri-inhom} and
\eqref{eq:bc-p_F-diri-inhom} on the other hand are typically strongly
enforced in the continuous formulation, and thus will be enforced
weakly in the discrete formulation using Nitsche's
method~\cite{Nitsche1971}. 

Starting from \eqref{eq:biot-total-press-I}, we multiply by a discrete test function $\bv \in \bV_h$ and integrate by parts over $\Omega$, incorporating 
\eqref{eq:bc-total-stress-inhom} naturally on $\Gamma_s$ and using Nitsche's method on $\Gamma_d$ to incorporate
\eqref{eq:bc-u-diri-inhom} to see that, $\forall \bv_h \in \bV_h$,
\begin{align*}
& (\mu\varepsilon(\bu), \varepsilon(\bv_h))_{\Omega} 
-(\mu\varepsilon(\bu)\cdot \bn, \bv_h)_{\Gamma_d}
-(\bu,\mu \varepsilon(\bv_h)\cdot \bn)_{\Gamma_d}
+ \dfrac{\gamma_u}{h}\mu(\bu, \bv_h)_{\Gamma_d}
-(p_{T}, \nabla\cdot \bv_h)_{\Omega}
+ (p_{T}, \bv_h \cdot \bn)_{\Gamma_d}
\\
& \quad = (\bf, \bv_h)_{\Omega} + (\bsigma_N, \bv_h)_{\Gamma_s} 
-(\bu_D, \mu\varepsilon(\bv_h)\cdot \bn)_{\Gamma_d}
+ \dfrac{\gamma_u}{h}\mu(\bu_D, \bv_h)_{\Gamma_d}.
 \end{align*}
Turning to \eqref{eq:biot-total-press-II}, we multiply by a discrete test function $q_{T,h} \in Q_{T,h}$
and add the consistent term $\left((\bu - \bu_D) \cdot \bn, q_{T,h}\right)_{\Gamma_d}$ for symmetry reasons, leading to, $\forall q_{T,h} \in Q_{T,h}$,
\begin{align*}
-(\nabla \cdot \bu, q_{T,h})_{\Omega} 
+ (\bu \cdot \bn, q_{T,h})_{\Gamma_d}
 - \frac{1}{\lambda} (p_{T}, q_{T,h})_{\Omega} +  \frac{1}{\lambda} (p_{F}, q_{T,h})_{\Omega} 
&= (\bu_D \cdot \bn, q_{T,h})_{\Gamma_d}.
\end{align*}
Finally, for \eqref{eq:biot-total-press-III} we multiply by a discrete test function $q_{F,h} \in Q_{F,h}$ 
and integrate by parts over $\Omega$. This time, we incorporate \eqref{eq:bc-p_F-neum-inhom} naturally on $\Gamma_d$, while \eqref{eq:bc-p_F-diri-inhom} is enforced weakly using Nitsche's method, obtaining, $\forall q_{F,h} \in Q_{F,h}$,
\begin{align*}
    \nonumber
&\frac{1}{\lambda} (p_{T}, q_{F,h})_{\Omega} 
    - \frac{2}{\lambda} (p_{F}, q_{F,h})_{\Omega} 
- (K \nabla p_{F}, \nabla q_{F,h})_{\Omega} \\
    \nonumber
    & \qquad
+ (K \nabla p_{F} \cdot \bn, q_{F,h})_{\Gamma_s}
+ (p_{F}, K \nabla q_{F,h} \cdot \bn )_{\Gamma_s}
- \dfrac{\gamma_p}{h} K (p_{F}, q_{F,h})_{\Gamma_s}
\\
&\quad
= (g, q_{F,h})_{\Omega}
- (g_N, q_{F,h})_{\Gamma_d} 
+ (p_{F,D}, K \nabla q_{F,h} \cdot \bn )_{\Gamma_s}
- \dfrac{\gamma_p}{h} K (p_{F,D}, q_{F,h})_{\Gamma_s}.    
\end{align*}
Next, set
\begin{align}
a_{1,h}(\bu, \bv_h)
&=
(\mu\varepsilon(\bu), \varepsilon(\bv_h))_{\Omega} 
-(\mu\varepsilon(\bu)\cdot \bn, \bv_h)_{\Gamma_d}
-(\bu, \mu\varepsilon(\bv_h)\cdot \bn)_{\Gamma_d}
+ \dfrac{\gamma_u}{h}\mu(\bu, \bv_h)_{\Gamma_d},
\label{defn-a1}\\
b_{1,h}(\bv, q_{T}) 
&=
-( \nabla\cdot \bv, q_{T})_{\Omega}
+ (\bv \cdot \bn, q_{T})_{\Gamma_d},
\label{defn-b1}\\
b_{2,h}(q_{F}, q_{T}) 
&= \dfrac{1}{\lambda}(q_{F}, q_{T})_{\Omega},
\label{defn-b2}\\
a_{2,h}( p_{T}, q_{T,h}) 
&= \dfrac{1}{\lambda}(p_{T}, q_{T,h})_{\Omega},
\label{defn-a2}\\
a^1_{3,h}(p_{F}, q_{F,h})
&=
(K \nabla p_{F}, \nabla q_{F,h})_{\Omega} 
- (K \nabla p_{F} \cdot \bn, q_{F,h})_{\Gamma_s}
- (p_{F}, K \nabla q_{F,h} \cdot \bn )_{\Gamma_s}
+ \dfrac{\gamma_p}{h} K (p_{F}, q_{F,h})_{\Gamma_s},
\label{defn-a3-1}\\
a^2_{3,h}(p_{F}, q_{F,h})
&=
\frac{2}{\lambda} (p_{F}, q_{F,h})_{\Omega}, \label{defn-a3-2}\\
a_{3,h}(p_{F}, q_{F,h})
&=
a^1_{3,h}(p_{F}, q_{F,h}) + a^2_{3,h}(p_{F}, q_{F,h}) \label{defn-a3}\\ 
L_1(\bv_h) 
&= 
(\bf, \bv_h)_{\Omega} + (\bsigma_N, \bv_h)_{\Gamma_s} 
-(\bu_D, \mu\varepsilon(\bv_h)\cdot \bn)_{\Gamma_d}
+ \dfrac{\gamma_u}{h}\mu(\bu_D, \bv_h)_{\Gamma_d},
\label{defn-L1}\\
L_2(q_{T,h})
&= (\bu_D \cdot \bn, q_{T,h})_{\Gamma_d}, \label{defn-L2}\\
L_3(q_{F,h})
&= (g, q_{F,h})_{\Omega} - (g_N, q_{F,h})_{\Gamma_d} + (p_{F,D}, K \nabla q_{F,h} \cdot \bn )_{\Gamma_s} - \dfrac{\gamma_p}{h} K (p_{F,D}, q_{F,h})_{\Gamma_s}. \label{defn-L3}
\end{align}
Note that if the solution $(\bu, p_{T}, p_{F})$ 
to \eqref{eq:biot-total-press-I}--\eqref{eq:bc-p_F-diri-inhom} is regular enough,
it satisfies the discrete weak formulation 
\begin{subequations} \label{eq:biot-total-press-weak-disc}
\begin{alignat}{4}
    &a_{1,h}(\bu, \bv_h) & + & \ b_{1,h}(\bv_h, p_{T}) & &= L_1(\bv_h), &  
    &\quad \forall \bv_h \in \bV_h, 
    \label{eq:biot-total-press-I-weak-disc}
    \\
    &b_{1,h}(\bu, q_{T,h}) \ & - & \ a_{2,h}(p_{T}, q_{T,h})  \ + & b_{2,h}(p_{F}, q_{T,h}) &= L_2(q_{T,h}), & &\quad \forall q_{T,h} \in Q_{T,h},
    \label{eq:biot-total-press-II-weak-disc}
    \\
    & & & \ b_{2,h}(q_{F,h}, p_{T}) \ - & \ a_{3,h}(p_{F}, q_{F,h}) &= L_3(q_{F,h}), & & \quad \forall q_{F,h} \in Q_{F,h}. 
    \label{eq:biot-total-press-III-weak-disc}
\end{alignat}
\end{subequations}

We introduce the corresponding scaled norms
\begin{align}
  \normv{\bv}^2 &= \mu\|\varepsilon( \bv)\|_{\Omega}^2 + \gamma_u \mu \|h^{-\onehalf} \bv\|_{\Gamma_d}^2, \label{Vh-norm} \\
 \normt{q_{T}}^2 &= \mu^{-1} \|q_{T}\|_{\Omega}^2, \\
 \normf{q_{F}}^2 &=  K \|\nabla q_{F}\|_{\Omega}^2  + \gamma_p K\|h^{-\onehalf} q_{F}\|_{\Gamma_s}^2 + \lambda^{-1} \|q_{F}\|_{\Omega}^2. \label{QFh-norm} 
\end{align}

\begin{remark}
The functionals defined in \eqref{Vh-norm} and \eqref{QFh-norm} are indeed norms, using Korn's inequality \cite[(1.19)]{Brenner-Korn} and Poincar\'e inequality \cite[(1.10)]{Brenner-Poincare}.
\end{remark}

In addition, we will also need the following norms, in which the error estimates will be obtained:
\begin{align}
  \normvast{\bv}^2 &= \normv{\bv}^2+ \mu\|h^{\onehalf} \bn \cdot \nabla \bv  \|_{\Gamma_d }^2, \label{eq:normstar_1}  \\
  \normtast{q_{T}}^2&= \normt{q_{T}}^2 +  \mu^{-1}\|h^{\onehalf} q_{T}  \|_{\Gamma_d}^2,   \\
  \normfast{q_{F}}^2 &= \normf{q_{F}}^2+ K\|h^{\onehalf} \bn \cdot \nabla q_{F}  \|_{\Gamma_s}^2, \label{eq:normstar_3}  \\
  \| (\bv,q_{T},q_{F} )\|^2_{*} &=  \normvast{\bv}^2 +\normtast{q_{T}}^2+ \normfast{q_{F}}^2.
\end{align}

The idea of the ghost penalty approach is to augment the formulation~\eqref{eq:biot-total-press-I-weak-disc}--\eqref{eq:biot-total-press-III-weak-disc} above with certain stabilization forms $g_{1,h}(\cdot, \cdot)$, $g_{2,h}(\cdot, \cdot)$, and $g_{3,h}(\cdot, \cdot) = g_{3,h}^1(\cdot, \cdot) + g_{3,h}^2(\cdot, \cdot)$ with associated seminorms $|\cdot|_{g_{i,h}}^2 = g_{i,h}(\cdot, \cdot)$, $i = 1,2,3$, acting in the vicinity of the embedded boundary so that the norms \eqref{Vh-norm}--\eqref{QFh-norm} are extended to the domain $\cT_h$ of the entire active mesh $\cT_h$, i.e.,
\begin{align}
  & \tnormv{\bv_h}^2 = \normv{\bv_h}^2 + |\bv_h|_{g_{1,h}}^2 \sim
  \mu\|\varepsilon( \bv_h)\|_{\cT_h}^2
  + \gamma_u\mu \|h^{-\onehalf} \bv_h\|_{\Gamma_d}^2, \label{defn-norm-Vh}\\
& \tnormt{q_{T,h}}^2 =\normt{q_{T,h}}^2  + |q_{T,h}|_{g_{2,h}}^2  \sim  \mu^{-1}\|q_{T,h}\|_{\cT_h}^2, \\
  & \tnormf{q_{F,h}}^2 = \normf{q_{F,h}}^2 + |q_{F,h}|_{g_{3,h}}^2   \sim \lambda^{-1} \|q_{F,h}\|_{\cT_h}^2 + K \|\nabla q_{F,h}\|_{\cT_h}^2
  + \gamma_p K\|h^{-\onehalf} q_{F,h}\|_{\Gamma_s}^2,  \label{defn-norm-Qf} \\
&\tn{(\bv_{h},q_{T,h},q_{F,h})} \tn_h^2 = \tnormv{\bv_h}^2 + \tnormt{q_{T,h}}^2 + \tnormf{q_{F,h}}^2. \label{defn-tri-norm}
\end{align}

To this end, we replace the  bilinear forms on the diagonal by their ghost-penalty enhanced counterparts
\begin{align*}
A_{1,h}(\bu_h, \bv_h) &= a_{1,h}(\bu_h, \bv_h)  + g_{1,h}(\bu_h, \bv_h),
\\
A_{2,h}(p_{T,h}, q_{T,h}) &=  a_{2,h}(p_{T,h}, q_{T,h}) + g_{2,h}(p_{T,h}, q_{T,h}),
\\
A^i_{3,h}(p_{F,h}, q_{F,h}) &=  a^i_{3,h}(p_{F,h}, q_{F,h}) + g^i_{3,h}(p_{F,h}, q_{F,h}), \ i = 1,2, \\
A_{3,h}(p_{F,h}, q_{F,h}) & = A^1_{3,h}(p_{F,h}, q_{F,h}) + A^2_{3,h}(p_{F,h}, q_{F,h}),
\end{align*}
resulting in the ghost penalty cut finite element method: find $(\bu_h, p_{T,h}, p_{F,h}) \in \bV_h \times Q_{T,h} \times Q_{F,h}$ such that
\begin{subequations}\label{eq:biot-total-press-weak-disc-ghost}
\begin{alignat}{4}
    &A_{1,h}(\bu_h, \bv_h) & + & \,b_{1,h}(\bv_h, p_{T,h}) & &= L_1(\bv_h), &  
    &\quad \forall \bv_h \in \bV_h,     \\
    &b_{1,h}(\bu_h, q_{T,h}) & \, - \, &A_{2,h}(p_{T,h}, q_{T,h})  & +  \,b_{2,h}(p_{F, h}, q_{T,h}) &= L_2(q_{T,h}), & &\quad \forall q_{T,h} \in Q_{T,h},
    \\
    & & &b_{2,h}(q_{F,h}, p_{T, h}) & \, - \, A_{3,h}(p_{F,h}, q_{F,h}) &= L_3(q_{F,h}), & & \quad \forall q_{F,h} \in Q_{F,h}. 
\end{alignat}
\end{subequations}
We define the unstabilized bilinear form
\begin{align}
  \label{eq:Bh-def}
  B_{h}&((\bu_h,p_{T,h},p_{F,h}),(\bv_h,q_{T,h},q_{F,h})) = a_{1,h}(\bu_h, \bv_h)  + b_{1,h}(\bv_h, p_{T,h}) + b_{1,h}(\bu_h, q_{T,h})  \nonumber\\ 
  &-  a_{2,h}(p_{T,h}, q_{T,h}) +  b_{2,h}(p_{F, h}, q_{T,h}) + b_{2,h}(q_{F,h}, p_{T, h})
  - a_{3,h}(p_{F,h}, q_{F,h}),
\end{align}
the ghost-penalty bilinear form
\begin{align}\label{ghost-form}
  G_{h}((\bu_h,p_{T,h},p_{F,h}),(\bv_h,q_{T,h},q_{F,h})) =  &g_{1,h}(\bu_h, \bv_h)  + g_{2,h}(p_{T,h}, q_{T,h}) + g^1_{3,h}(p_{F,h}, q_{F,h}) + g^2_{3,h}(p_{F,h}, q_{F,h}),  
\end{align}
and the total bilinear form 
\begin{align}
  A_{h}((\bu_h,p_{T,h},p_{F,h}),(\bv_h,q_{T,h},q_{F,h})) = &B_{h}((\bu_h,p_{T,h},p_{F,h}),(\bv_h,q_{T,h},q_{F,h}))
  + G_{h}((\bu_h,p_{T,h},p_{F,h}),(\bv_h,q_{T,h},q_{F,h})).
  \label{eq:total_bilinear_form}
\end{align}
Method \eqref{eq:biot-total-press-weak-disc-ghost} can be rewritten as: find $(\bu_h, p_{T,h}, p_{F,h}) \in \bV_h \times Q_{T,h} \times Q_{F,h}$ such that for all $(\bv_h, q_{T,h}, q_{F,h}) \in \bV_h \times Q_{T,h} \times Q_{F,h}$,
\begin{equation}\label{method-v2}
A_{h}((\bu_h,p_{T,h},p_{F,h}),(\bv_h,q_{T,h},q_{F,h})) = L_1(\bv_h) + L_2(q_{T,h}) + L_3(q_{F,h}).
\end{equation}

\section{Useful inequalities, approximation properties and ghost penalty design}
\label{sec:preliminaries}
To prepare for the stability and error analysis in Section~\ref{sec:analysis}, we here collect some useful
inequalities, discuss the construction and properties of various quasi-interpolation
operators, and review the design of suitable ghost penalties for the Biot problem.

\subsection{Trace inequalities and inverse estimates} 
We recall the local trace inequalities for a function $v \in H^1(T)$,
\begin{align}
  \| v \|_{\partial T} &\lesssim h_T^{-\onehalf} \norm{v}_T + h_T^{\onehalf} \norm{\nabla v}_T, \quad \forall T \in \cT_h, \\
  \| v \|_{\Gamma \cap T} &\lesssim h_T^{-\onehalf} \norm{v}_T + h_T^{\onehalf} \norm{\nabla v}_T, \quad \forall T \in \cT_h \label{eq:trace-ineq-2}.
\end{align}
In addition, we have the following inverse estimates for a discrete function $v_h \in \bbP^k(T)$:
\begin{align}
  \label{eq:inverse-estimates-I}
  \norm{\nabla v_h}_T &\lesssim h_T^{-1} \norm{v_h}_T.
\end{align}
\subsection{Quasi-interpolation operators}
To obtain the required interpolation operator, we first define the polynomial space \(V_h = X_h^{k}\) and let $\pi^{*}_{h}:L^2(\cT_h) \rightarrow V_h$ be the standard Scott--Zhang interpolation operator \cite{Scott-Zhang},
which for all \(T \in \cT_{h}\) satisfies the stability and interpolation error estimates
\begin{alignat}{3}
| \pi^{*}_{h} v |_{r,T} &\lesssim \| v \|_{r,\cP(T)},
\quad &0\leqslant r &\leqslant k + 1,
 \label{SZ-1} \\
\| v - \pi^{*}_{h} v \|_{r,T} &\lesssim h^{s-r}| v |_{s,\cP(T)},
\quad &0\leqslant r \leqslant s &\leqslant k + 1,
 \label{SZ-2} \\
\| v - \pi^{*}_{h} v \|_{r,\Gamma\cap T} &\lesssim h^{s-r-\onehalf}| v |_{s,\cP(T)}, 
\quad & 0\leqslant r + \frac{1}{2} \leqslant s &\leqslant k + 1, \label{SZ-3}
\end{alignat}
where $\cP(T)$ is the patch of neighbors of element $T$, meaning the domain including all elements sharing at least one vertex with $T$. We further
define $\pi_{h}:L^2(\Omega) \rightarrow V_{h}$ as
\begin{equation}
  \label{eq:ext-interpolant-def}
\pi_{h} v = \pi^{*}_{h} \cE v,
\end{equation}
where \(\cE :H^s(\Omega) \rightarrow H^s(\bbR^n)\), $s\geqslant 0$, is a bounded extension operator such that
\begin{equation} \label{stab-ext}
\| \cE u \|_{s,\bbR^n} \lesssim\| u \|_{s,\Omega},
\end{equation}
see \cite{Stein1970}. Estimates \eqref{SZ-1}--\eqref{SZ-3} imply
\begin{alignat}{3}
\| \pi_{h} v \|_{r,T} &\lesssim \| \cE v \|_{r,\cP(T)},
\quad &0\leqslant r &\leqslant k + 1,
 \label{interp-stab} \\
\| v - \pi_{h} v \|_{r,T\cap\Omega} &\lesssim h^{s-r}| \cE v |_{s,\cP(T)},
\quad & 0\leqslant r \leqslant s &\leqslant k + 1,
 \label{interp-T} \\
\| v - \pi_{h} v \|_{r,\Gamma\cap T} &\lesssim h^{s-r-\onehalf}| \cE v |_{s,\cP(T)}, 
\quad & 0\leqslant r + \frac{1}{2} \leqslant s &\leqslant k + 1, \label{interp-G}
\end{alignat}
The above bounds, together with the stability of the extension operator \( \cE \) \eqref{stab-ext}, imply the following interpolation estimates for the norms \eqref{eq:normstar_1}--\eqref{eq:normstar_3}.

\begin{lemma} [Interpolation estimates] Assume that \( \bv \in [H^r(\Omega)]^n \), $r \geqslant 2$, \(q_{T} \in H^s(\Omega) \), $s \geqslant 1$, \(  q_{F} \in H^t(\Omega)\), $t \geqslant 2$, and let \( \pi_{h} \bv \in [X_h^k]^n\),  \( \pi_{h} q_T \in X_h^m \),  \( \pi_{h} q_F \in  X_h^l \). Then for \(\bar{r} = \min\{k+1,r\}\), \(\bar{s} = \min\{m+1,s\}\), and \(\bar{t} = \min\{l+1,t\}\) the following estimates hold:
\begin{subequations} \label{eq:energy_interpolation}
\begin{align}
\normvast{\bv - \pi_{h} \bv} &\lesssim \mu^{\onehalf} h^{\bar{r}-1}  |\bv|_{\bar{r},\Omega}, \\
  \normtast{q_{T} - \pi_{h} q_T} &\lesssim \mu^{-\onehalf} h^{\bar{s}}  |q_{T} |_{\bar{s},\Omega}, \\
  \normfast{q_{F} - \pi_{h} q_F} &\lesssim \left( K ^{\onehalf} +\lambda^{-\onehalf} h \right)  h^{\bar{t}-1} |q_{F}|_{\bar{t},\Omega}.
\end{align}
\end{subequations}
\end{lemma}

\begin{figure}[t]
    \centering
    \includegraphics[trim = {0cm 0cm 0cm 13cm}, clip, width=0.4\textwidth]{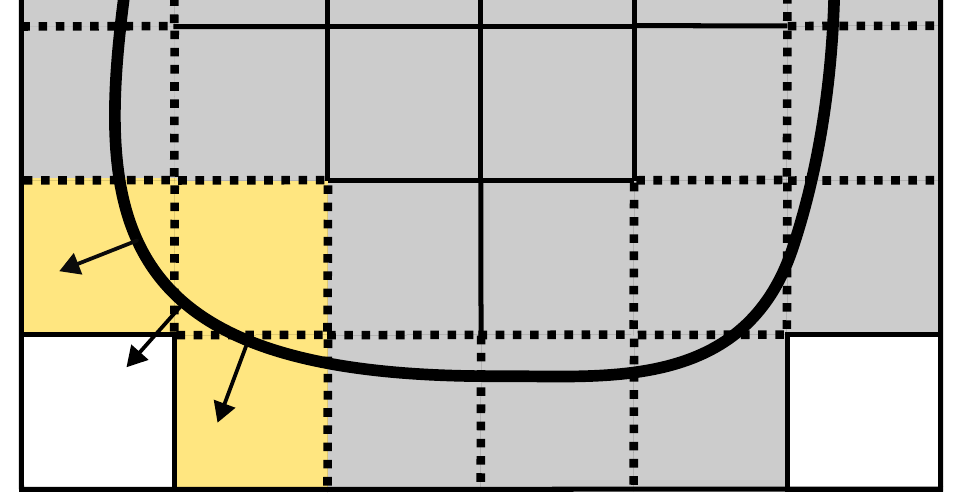}
    \caption{Example of a patch, with non-constant normal vector.} 
    \label{fig:cutfem-patches}
\end{figure}

We next construct a modified Scott--Zhang interpolant with an additional orthogonality property on $\Gamma_s$, which will be needed for the proof of stability of the method. We follow the approach from \cite{BeckerBurmanHansbo2009}. Let $\cT_h^{\Gamma_s}$  be 
the set of elements that intersect the boundary $\Gamma_s$,
$$
\cT_h^{\Gamma_s} = \{ T \in \cT_{h}  \mid  T \cap \Gamma_s  \neq \emptyset  \},
$$
and consider the set of all elements in $\cT_h^{\Gamma_s}$
and elements that share a face or vertex with
these elements,
\begin{align*}
  \overline{\cT}_h^{\Gamma_s} = \{ T \in \cT_h \, | \, T \cap T'  \neq \emptyset ,\, T' \in \cT^{\Gamma_s}_h \}.
 \end{align*} 
Next, we regroup the elements in \(\overline{\cT}_h^{\Gamma_s}\) into non-overlapping patches \( \{ \cP_i \}_{i=1}^{N}\) such that for each patch, a patch function \(\phi_i\) can be constructed from the basis functions with support in \(\cP_i\). Further, \(\cP_i\) and \(\phi_i\) have the properties 
\begin{enumerate}
  \item \( h \lesssim \diam(\cP_i) \lesssim h \),
  \item \( h^{n-1} \lesssim \displaystyle\int_{\Gamma_s \cap \cP_i}  \phi_i \lesssim h^{n-1} \),
  \item \( h^{-1} \lesssim | \nabla \phi_i (x) | \lesssim h^{-1} \). 
\end{enumerate}

In contrast to the proof given in
\cite{BeckerBurmanHansbo2009}, we allow for \emph{non-constant}
normal vectors on each patch, see Figure~\ref{fig:cutfem-patches} for
an example of such a patch. This is particularly important if we want to
consider the use of independently generated surface meshes or higher-order implicit
descriptions to represent the embedded surface.
To take patch-wise normal-field variations into account, 
we recall that the boundary \( \Gamma_s \) is assumed to be of class
\( C^2 \), and we therefore know that there exists an
\(\varepsilon-\)tubular neighborhood \( U_\varepsilon(\Gamma_s) \),
and \(\varepsilon_0 > 0\) such that there is a unique closest point \(
p(x) \) on \( \Gamma_s \) for all \( x \in U_{\varepsilon_0}(\Gamma_s)
\), where \( \abs{x - p(x)} = \dist\left(\Gamma_s,x\right)\). Choosing
\(h\) small enough, we have that \(\overline{\cT}_h^{\Gamma_s}\) is
contained in \( U_{\varepsilon_0}(\Gamma_s) \). Choosing \( \bx_0 \in
\cP_i\cap\Gamma_s \), the following bound holds for the normal vector:
\begin{align}
\norm{\bn(\bx) - \bn(\bx_0)}_{\infty,\cP_i\cap\Gamma_s } \leqslant \norm{\nabla \bn}_{\infty,\cP_i} \norm{\bx - \bx_0}_{\infty,\cP_i\cap\Gamma_s }  \leqslant c_I h, \label{eq:normal_bound}
\end{align}
so that \( c_I \) is the maximum of the principal curvatures. 

We want the modified interpolant, which we denote by $\pi^c_{h}$, to satisfy the following orthogonality property:
\begin{align}
  \int_{\cP_i \cap \Gamma_s} (\bv - \pi^c_{h}\bv )\cdot \bn \,\d s= 0 \quad \forall \, \cP_i. \label{eq:normal_ort_prop}
\end{align}
In \cite{BeckerBurmanHansbo2009} the normal was implicitly assumed to be constant on each patch. In the following, we extend the construction to non-constant normal vectors and dimension \( n \).  

\begin{lemma}
 There exists a modified Scott--Zhang interpolation operator \( \pi^c_{h}:H^1(\Omega) \rightarrow V_h \) that satisfies
 for \( \forall \bv \in [H^r(\Omega)]^n\), \( r \geqslant 1 \),
 the orthogonality property \eqref{eq:normal_ort_prop} and the estimates
\begin{align}
  \norm{ \pi_h\bv - \pi^c_{h} \bv}_{\cT_h} + h^{\onehalf}\norm{\bv - \pi^c_{h} \bv}_{\Gamma}  &\lesssim h \| \bv \|_{1,\Omega}, \label{interpolation_stab2}  \\
\norm{\pi^c_{h} \bv}_{1,\cT_h} & \lesssim \| \bv \|_{1,\Omega}. \label{eq:stab-pi-c}
\end{align}
\end{lemma}

\begin{proof}

Let \( \bY_i \) be the (vector-valued) degree of freedom associated with \(\phi_i\) 
defined on a patch~\(\cP_i\).  
Let \( \be^1(\bx) = \bn(\bx) \) be the surface normal, and define \( \be^k(\bx)\) for \(k=2,\dots,n\) to be orthogonal vectors such that they span the tangential space on \( \Gamma_s \) at $\bx$. 
We determine \( \bY_i \) from the linear system
\begin{align}\label{Y-defn}
  \int_{\cP_i \cap \Gamma_s} \bY_i \cdot \be^k \phi_i \,\d s &= \int_{\cP_i \cap \Gamma_s} (\bv - \pi_{h}\bv )\cdot \be^k\,\d s \qquad \text{for } k = 1,\dots,n,
\end{align}
which we can write in the matrix equation
\begin{align*}
  A_i \bY_i = \bV_i,
\end{align*}
with 
\begin{align*}
  A_i = 
  \begin{bmatrix} 
    \int_{\cP_i \cap \Gamma_s} \phi_i (\be^1)^T \,\d s  \\
    \vdots  \\
    \int_{\cP_i \cap \Gamma_s} \phi_i (\be^n)^T \,\d s
    \end{bmatrix}, \quad
  \bV_i =   \begin{bmatrix} 
    \int_{\cP_i \cap \Gamma_s}  (\bv - \pi_{h}\bv )\cdot \be^1 \,\d s  \\
    \vdots  \\
    \int_{\cP_i \cap \Gamma_s}  (\bv - \pi_{h}\bv )\cdot \be^n \,\d s
    \end{bmatrix}.
\end{align*}
We next show that $A_i$ is invertible. To that end, let \( \bar{\be}^k_{i} =  \be^k(\bx_0) \) for an \( \bx_0 \in \cP_i \cap \Gamma_s \), and define the matrix \(B_i\) 
\begin{align*}
  B_i =  \int_{\cP_i \cap \Gamma_s}  \phi_i \,\d s \,
    \begin{bmatrix} 
      (\bar{\be}^1)^T  \\
      \vdots  \\
      (\bar{\be}^n)^T
      \end{bmatrix}.
\end{align*}
From construction, \( B_i \)  is an invertible matrix. Its Euclidean-based matrix norm can easily be estimated as \( \abs{B_i}_{\bbR^n} \sim h^{n-1} \), and it follows that \(  \abs{B_i^{-1}}_{\bbR^n} \lesssim h^{1-n}.\) We also define the matrix \(C_i\), with elements \(C^{k,j}_i = \int_{\cP_i \cap \Gamma_s} (e^{k,j} - \bar{e}^{k,j}_i) \phi\,\d s\), where \( e^{k,j}\) is the \(j\)-th component of \(\be^k\). 
Recall the bound of the normal vector \eqref{eq:normal_bound}, which is also valid for the orthogonal vectors, so that for one element of \(C_i\) 
\begin{align*}
  \abs{C_i^{k,j}} \leqslant  \norm{e^{k,j}_i - \bar{e}^{k,j}_i}_{\infty, \cP_i \cap \Gamma_s} \int_{\cP_i \cap \Gamma_s} |\phi|\,\d s \lesssim h^{n},
\end{align*}
giving \( \abs{C_i}_{\bbR^n}  \lesssim h^{n} \). 
We can write \(  A_i = B_i + C_i = B_i(\cI + B_i^{-1} C_i),\) and in addition we have
\begin{align}
  \abs{B^{-1}_i C_i}_{\bbR^n} \leqslant   \abs{B^{-1}_i}_{\bbR^n}   \abs{C_i}_{\bbR^n}  \lesssim h,  \label{eq:proof_interpolant_bound}
\end{align} meaning that for small \( h \), \( A_i \) is invertible and \(\bY_i\) is given by \(  \bY_i = A_i^{-1}\bV_i \). The modified interpolant is defined by
\begin{align}
  \pi^c_{h}\bv = \pi_{h}\bv + \sum^{N}_{i = 1} \phi_i \bY_i,  \label{eq:modiefied_interpolant}
\end{align}
which, combined with \eqref{Y-defn}, implies that the orthogonality property \eqref{eq:normal_ort_prop} holds.

To show \eqref{interpolation_stab2}, we next establish a bound for \(\abs{\bY_i}_{\bbR^n} \leqslant  \abs{A^{-1}}_{\bbR^n} \abs{\bV_i}_{\bbR^n}\), and begin by finding a bound for \(  \abs{A^{-1}}_{\bbR^n}  \), 
\begin{align*}
  \abs{A^{-1}}_{\bbR^n}  = \abs{(\cI + B_i^{-1} C_i)^{-1}B_i^{-1}}_{\bbR^n} \lesssim \abs{(\cI + B_i^{-1} C_i)^{-1}}_{\bbR^n} \abs{B_i^{-1}}_{\bbR^n}. 
\end{align*}
For small $h$ we can use power series expansion together with \eqref{eq:proof_interpolant_bound} to get 
\begin{align}
  \abs{A^{-1}}_{\bbR^n} &\leqslant \abs{(\cI + B_i^{-1} C_i)^{-1}}_{\bbR^n} \abs{B_i^{-1}}_{\bbR^n} = \left\lvert \sum_0^{\infty}(- B_i^{-1}C_i)^k \right\rvert_{\bbR^n} \abs{B_i^{-1}}_{\bbR^n} \nonumber \\  &\leqslant \sum_0^{\infty} \abs{B_i^{-1}C_i}_{\bbR^n}^k \abs{B_i^{-1}}_{\bbR^n}  
  = \frac{\abs{B_i^{-1}}_{\bbR^n}  }{1-\abs{B_i^{-1}C_i}_{\bbR^n} }  \lesssim \frac{\abs{B_i^{-1}}_{\bbR^n}  }{1-h}\lesssim h^{1-n} \label{eq:A_inverse}.
\end{align}
Next, we apply the Cauchy-Schwarz inequality and \eqref{interp-G} to obtain
\begin{align*}
  \abs{V^k_i} = \left| \int_{\cP_i \cap \Gamma_s} (\bv - \pi_{h}\bv )\be^k \,\d s \right|  
  \lesssim h^{(n-1)/2} \norm{\bv - \pi_{h}\bv}_{\cP_i \cap \Gamma_s}  
  \lesssim h^{n/2} \abs{\cE \bv}_{1,\cP(\cP_i)},
\end{align*}
where  \( \cP(\cP_i) \) consists of all the patches that share a vertex with \( \cP_i \). 
Then, combining this with \eqref{eq:A_inverse} we get the following bound for \(\bY_i\):
\begin{align}\label{bound-Y}
  \abs{\bY_i}_{\bbR^n} \leqslant\abs{A^{-1}}_{\bbR^n} \abs{\bV_i}_{\bbR^n}\lesssim h^{-n/2+ 1} \abs{\cE \bv}_{1,\cP(\cP_i)}. 
\end{align}
Let $\cI_i$ be the index set of patches that overlap with $\cP_i$. 
Using the definition of $\pi^c_{h}$ \eqref{eq:modiefied_interpolant} and the bound on \(\bY_i\) \eqref{bound-Y}, for any \( T \in \cP_i \) we have that
\begin{align*}
  \norm{\pi_h\bv - \pi^c_{h} \bv}_{T} \leqslant \sum_{j \in \cI_i}\norm{\phi_j \bY_j}_{T} \lesssim h^{n/2} \sum_{j \in \cI_i} \abs{\bY_j}_{\bbR^n}
\lesssim h \sum_{j \in \cI_i} \abs{\cE \bv}_{1,\cP(\cP_j)}.
\end{align*}
Similarly, using \eqref{interp-G}, we obtain
\begin{align*}
  \norm{\bv - \pi^c_{h} \bv}_{\Gamma\cap T} &\leqslant \norm{ \bv- \pi_{h} \bv}_{\Gamma\cap T} +  \sum_{j \in \cI_i}\norm{\phi_j \bY_j}_{\Gamma\cap T} 
  \lesssim h^{\onehalf}|\cE \bv|_{1,\cP(\cP_i)} + h^{(n-1)/2} \sum_{j \in \cI_i}
  \abs{\bY_j}_{\bbR^n}
   \lesssim h^{\onehalf} \sum_{j \in \cI_i} \abs{\cE \bv}_{1,\cP(\cP_j)}
\end{align*}
and, using \eqref{interp-stab}, we have
\begin{align}
  \norm{\pi^c_{h} \bv}_{1,T} \leqslant \norm{\pi_{h}}_{1,T} +  \sum_{j \in \cI_i} \norm{\phi_j \bY_j}_{1,T} 
  \lesssim  \norm{\cE \bv}_{1,\cP(\cP_i)} +  h^{n/2-1}\sum_{j \in \cI_i}
\abs{\bY_j}_{\bbR^n} 
\lesssim \sum_{j \in \cI_i} \norm{\cE \bv}_{1,\cP(\cP_j)}.
  \label{eq:bound_grad_pic} 
\end{align}
Noting that the above three bounds hold trivially for $T \in \cT_h \setminus \overline{\cT}_h^{\Gamma_s}$, the estimates \eqref{interpolation_stab2} and \eqref{eq:stab-pi-c} follow from summing them over $T \in \overline{\cT}_h^{\Gamma_s}$, using that by construction the number of patch overlaps is bounded, and applying \eqref{stab-ext}.
\end{proof}

\subsection{Role and design of the ghost penalties}
The total pressure formulation \eqref{eq:biot-total-press-weak-disc-ghost} 
can be seen as a perturbed Stokes problem for the displacement and total pressure that is intertwined with a Poisson problem for the fluid pressure.
Thus, to devise suitable ghost penalty stabilization, we have to take into
account the different roles played by the individual bilinear form contributions.
To do so, we start our discussion from an abstract finite element
space $\cV_h$ satisfying 
$\mathbb{P}_{\mathrm{c}}^k(\cT_h) \subset \cV_h$
and review the required properties for $H^1$ stabilizing
ghost penalty $g_h(\cdot, \cdot)$ suitable for cut finite element formulations of
elliptic problems, see e.g.~\cite{GuerkanMassing2019,dePrenterVerhooselvanBrummelenEtAl2023}.
Then, departing from this abstract setting, we discuss the specific ghost penalties needed
for the relevant bilinear forms appearing in~\eqref{eq:Bh-def}.

\subsubsection{H1-stabilizing ghost penalties and realizations}
We assume that the ghost penalty $g_h(\cdot, \cdot)$ is a symmetric and positive semi-definite bilinear form on $\cV_h$
satisfying the following properties:

\begin{enumerate}[label=\textbf{A\arabic*}]
\item  ($H^1$-seminorm extension property for $v_h \in \cV_h$)
  \begin{align*}
    \| \nabla v_h \|_{\cT_h}
    \lesssim
    \| \nabla v_h \|_{\cT_h^i}  + |v_h|_{g_h},
  \end{align*}
\item  (Weak consistency for $v \in H^s(\Omega)$)
  \begin{align*}
    |\pi_h v |_{g_h} \lesssim h^{r-1} \| v\|_{r,\Omega},
  \end{align*}
  where $r = \min\{s,k+1\}$ and
  $\pi_h$ is the quasi-interpolation operator given in~\eqref{eq:ext-interpolant-def}.
\end{enumerate}
These two properties are crucial for deriving 
geometrically robust discrete coercivity and optimal a priori error estimates 
for cut finite element formulations of Poisson-type problems, see \cite{GuerkanMassing2019,dePrenterVerhooselvanBrummelenEtAl2023}.
To ensure that the associated system matrices satisfy the usual condition number bounds, 
we typically require the following additional properties to be satisfied by the ghost penalty $g_h(\cdot, \cdot)$:

\begin{enumerate}[label=\textbf{A\arabic*}]
\setcounter{enumi}{2}
\item  ($L^2$-norm extension property for $v_h \in \cV_h$)
  \begin{align*}
    \| v_h \|_{\cT_h} 
    &\lesssim
      \| v_h \|_{\cT_h^i} + | v_h |_{g_h},
  \end{align*}
\item  (Inverse inequality for $v_h \in \cV_h$),
  \begin{align*}
    |v_h|_{g_h} \lesssim h^{-1} \| v_h \|_{\cT_h}.
  \end{align*}
\end{enumerate}

Next, we briefly review the most common realizations of the ghost
penalty $g_h(\cdot,\cdot)$ satisfying
Assumptions~\textbf{A1}--\textbf{A4} for $\cV_h$.
\emph{Facet-based} ghost penalties~\cite{BeckerBurmanHansbo2009,Burman2010,Massing-etal-cutfem-Stokes}
\begin{align}
  g_h^f(v,w) &=  \sum_{j=0}^p \gamma^f_p h^{2j-1} 
  (\jump{\partial_n^j v},  \jump{\partial_n^j w})_{\cF_h^g},
  \label{eq:gc_face_based}
\end{align}
penalize the jumps of
normal-derivatives of all relevant polynomial orders across faces belonging to the
facet set 
\begin{align*}
  \cF_h^g
  = \{ F \in \cF_h:  T^+ \cap \Gamma \neq \emptyset \lor T^- \cap \Gamma \neq \emptyset \}.
\end{align*}
Here we use
the notation $\partial_n^j v \coloneqq \sum_{| \alpha | =
j}\tfrac{D^{\alpha} v(x)n^{\alpha}}{\alpha !}$ for multi-indices $\alpha
= (\alpha_1, \ldots, \alpha_d)$, $|\alpha| = \sum_{i} \alpha_i$ and
$n^{\alpha} = n_1^{\alpha_1} n_2^{\alpha_2} \cdots n_d^{\alpha_d}$.
The constant $\gamma^f$ denotes a dimensionless stability parameter.
Note that for $H^1$-conforming finite element spaces, the jump terms for $j=0$ vanish.

In \cite{Burman2010}, ghost penalties  based on a
\emph{local projection stabilization} were proposed. For a given
patch $P$ with $\diam(P) \lesssim h$ containing the two elements $T_1$
and $T_2$, one defines the $L^2$-projection
$\pi_{P} : L^2(P) \to \mathbb{P}_{p}(P)$ onto the space of polynomials of
order $p$ associated with the patch~$P$. For $v \in V_h$,
the fluctuation operator $\kappa_P = I - \pi_P$ measures then the deviation of
$v|_P$ from being a polynomial defined on $P$.
By choosing certain patch definitions, 
a coupling between
elements with a possible small cut and an interior element
is ensured. One patch choice arises naturally from the definition
of $\cF_h^g$ by defining
the patch $P(F) = T^+_F \cup T^-_F$ for two elements $T^+_F$, $T^-_F$ sharing the interior face $F$ and
setting
\begin{align*}
  \cP_1 = \{P(F) \}_{F\in \cF_h^g}.
\end{align*}
A second possibility is to use neighborhood patches $\omega(T)$,
\begin{align*}
  \cP_2 = \{\omega(T) \}_{T \in \cT^{\Gamma}_h}.
\end{align*}
Alternatively, one can mimic the cell agglomeration approach taken in
classical unfitted finite element methods
approaches~\cite{SollieBokhoveVegt2011,JohanssonLarson2013,BadiaVerdugoMartin2018}
by associating to each cut element $T \in \cT_h^{\Gamma}$ with a small
intersection $|T \cap \Omega|_d \ll |T|_d$ an element
$T' \in \cT_h^i$ such that
the ``agglomerated
patch'' $P_a(T) = T \cup T'$ satisfies $\diam{P_a(T)} \lesssim h$.
A proper collection of patches
is then given by
\begin{align*}
  \cP_3 = \{P_a(T)\}_{T\in\cT_h^{\Gamma}}.
\end{align*}
The resulting \emph{local projection stabilization} ghost penalties  are then defined as follows:
\begin{align*}
  g_h^p(v,w) &= \gamma^p \sum_{P \in \cP} h^{-2} ( \kappa_P v,  \kappa_P w)_P, \quad \cP \in \{\cP_1, \cP_2, \cP_3\}.
\end{align*}

Finally, an elegant version of a patch-based ghost penalty
avoiding the use of local projection operators 
was presented in~\cite{Preuss2018}.
By extending polynomials $u_i$ defined on an element $T_i$ naturally to global polynomials $u_i^e$,
a volume-based jump on a patch $P = T_1 \cup T_2$ can be defined by
$\jump{u}_P = u_1^e - u_2^e$ which give rise to the \emph{volume based} ghost penalty
\begin{align*}
  g_h^v(v,w) &= \gamma_b^v \sum_{P \in \cP_1} h^{-2} (\jump{v}_P,\jump{w}_P)_P, \quad \cP \in \{\cP_1, \cP_2, \cP_3\}.
\end{align*}
More realizations with alternative penalty operators and patch choices
can be found in the literature, see e.g.
\cite{dePrenterVerhooselvanBrummelenEtAl2023} and the references
therein.

\subsubsection{Ghost penalties for the Biot problem}
Starting from the abstract setting, we define the individual ghost penalties for the relevant bilinear forms appearing in~\eqref{eq:Bh-def} as follows:
\begin{alignat*}{2}
  g_{1,h}(\bu_h, \bv_h) &= \mu g_h(\bu_h, \bv_h), && \quad \forall \bu_h, \bv_h \in \bV_h,
  \\
  g_{2,h}(p_{T,h}, q_{T,h}) &= \frac{h^2}{\mu}  g_h(p_{T,h}, q_{T,h}), && \quad \forall p_{T,h}, q_{T,h} \in Q_{T,h},
  \\
  g_{3,h}^1(p_{F,h}, q_{F,h}) &= K g_h(p_{F,h}, q_{F,h}), && \quad \forall p_{F,h}, q_{F,h} \in Q_{F,h},
  \\
  g_{3,h}^2(p_{F,h}, q_{F,h}) &= \frac{1}{\lambda} g_h(p_{F,h}, q_{F,h}), && \quad \forall p_{F,h}, q_{F,h} \in Q_{F,h}.
\end{alignat*}
Here, we have tacitly assumed that $g_h(\cdot,\cdot)$ can be defined analogously for vector-valued 
arguments $\bu_h, \bv_h \in \bV_h$.

The scaling of the individual ghost penalties is chosen to match the scaling of the corresponding bilinear forms
in~\eqref{eq:Bh-def}. In particular, the $h^2$-scaling appearing in $g_{2,h}(\cdot,\cdot)$
reflects the fact that the pressure $p_{T,h}$ needs to be controlled in an $L^2$-like norm,
and later we will see that the $h$-scaled $H^1$-seminorm  $\|h\nabla
p_{T,h}\|_{\cT_h}$ of the pressure naturally appears in the analysis of the
discrete inf-sup condition, see Assumption~\ref{Assumption_GO}.
To prepare for this, we record here some important properties of the
pressure ghost penalty $g_{2,h}(\cdot,\cdot)$, which follow directly from
Assumptions \textbf{A1}--\textbf{A4} for the abstract ghost penalty $g_h(\cdot,\cdot)$.
First, using Assumption \textbf{A1}, we see that
\begin{align}
  \|h\nabla p_{T,h} \|_{\cT_h} \lesssim \|h\nabla p_{T,h} \|_{\cT_h^i} + \mu^{\onehalf}|p_{T,h}|_{g_{2,h}} \lesssim \| p_{T,h} \|_{\cT_h}, 
  \quad \forall p_{T,h} \in Q_{T,h}, \label{eq:gp_extendpt}
\end{align}
where the second inequality follows from  \textbf{A4} and the inverse estimate~\eqref{eq:inverse-estimates-I}.
Second, using Assumption \textbf{A2}, we have the weak consistency estimate
\begin{align*}
  |\pi_h p_{T}|_{g_{2,h}} \lesssim \mu^{-\onehalf} h^{\bar{s}} \| p_{T} \|_{\bar{s},\Omega}, \quad \forall p_{T} \in H^s(\Omega),
\end{align*}
where $\pi_h: L^2(\Omega) \to Q_{T,h}$ is the quasi-interpolation operator 
given in~\eqref{eq:ext-interpolant-def} with $V_h = X_h^{l}$ and $\bar{s} = \min\{s,l+1\}$.
Third, we show that the pressure ghost penalty $g_{2,h}(\cdot,\cdot)$ satisfies an $L^2$-norm extension property.
\begin{lemma}  For the pressure ghost penalty \(g_{2,h}(\cdot,\cdot)\) defined above, the following estimate holds:  \label{lemma:l2-extension-pt}
\begin{align}
  \|p_{T,h} \|_{\cT_h} \lesssim \|p_{T,h} \|_{\cT_h^i} + \mu^{\onehalf}|p_{T,h}|_{g_{2,h}} 
  \quad \forall p_{T,h} \in Q_{T,h} \label{eq:l2_stab_pt}. 
\end{align}
\end{lemma}
\begin{proof}
We start by decomposing the mesh $\cT_h$ into the interior mesh and the mesh cut by \(\Gamma\), 
\begin{align}
   \|p_{T,h} \|_{\cT_h} &\leqslant \|p_{T,h} \|_{\cT^i_h} + \|p_{T,h} \|_{\cT^\Gamma_h} \label{eq:proof_l2_stab_pt_I}.
\end{align}
Next,  let \( \{ \cP_j \}_{j=1}^{M}\) be a finite collection of patches that cover  \( \cT^\Gamma_h \), where each patch satisfies \( \diam(\cP_j) \lesssim h\) and also contains at least one element from \(\cT^i_h\). Let \( \bar{p}_{T,h}^j \) be the average of \( p_{T,h} \) over the patch \( \cP_j \). We have
\begin{align}
    \|p_{T,h} \|_{\cT^\Gamma_h} 
    \leqslant \sum_{j=1}^{M} \|p_{T,h} \|_{\cP_j} 
    \leqslant \sum_{j=1}^{M} \|p_{T,h} - \bar{p}_{T,h}^j \|_{\cP_j} + \sum_{j=1}^{M}\|\bar{p}_{T,h}^j \|_{\cP_j} = \text{I} + \text{II} \label{eq:proof_l2_stab_pt_II}. 
\end{align}
Using a local Poincar\'e--Friedrichs inequality on each patch, the inverse estimate~\eqref{eq:inverse-estimates-I}, and Assumption \textbf{A1}, we can bound the first term as
\begin{align}
  \text{I} \lesssim \sum_{j=1}^{M} h \|\nabla p_{T,h} \|_{\cP_j} 
  \leqslant  \|h\nabla p_{T,h} \|_{\cT_h^\Gamma} + \|h\nabla p_{T,h} \|_{\cT_h^i} 
  \lesssim \|h\nabla p_{T,h} \|_{\cT_h^i} +  \mu^{\onehalf} |p_{T,h}|_{g_{2,h}} 
  \lesssim \| p_{T,h} \|_{\cT_h^i} +   \mu^{\onehalf}|p_{T,h}|_{g_{2,h}}   \label{eq:proof_l2_stab_pt_III}.
\end{align}
For the second term, let \(\cP_j^i = \cP_j\cap \cT_h^i \) denote the part of the patch \(\cP_j\) that belongs to the interior mesh. Since \( \diam(\cP_j) \lesssim h\) and $\cP_j^i$ contains at least one element, it follows that $\abs{\cP_j} \lesssim \abs{\cP_j^i}$. We then have
\begin{align}
  \text{II} \lesssim \sum_{j=1}^{M} \| \bar{p}_{T,h}^j \|_{\cP_j^i} 
  \lesssim  \sum_{j=1}^{M} \| \bar{p}_{T,h}^j - p_{T,h}\|_{\cP_j^i} + \sum_{j=1}^{M} \| p_{T,h}\|_{\cP_j^i}
  \lesssim  \sum_{j=1}^{M} \| h \nabla p_{T,h} \|_{\cP_j^i} + \| p_{T,h}\|_{\cT_h^i}
  \lesssim \| p_{T,h}\|_{\cT_h^i},  \label{eq:proof_l2_stab_pt_IIV}
\end{align}
where we have used a local Poincar\'e--Friedrichs inequality on each patch in the next-to-last step, and the inverse estimate~\eqref{eq:inverse-estimates-I} in the last step. Combining \eqref{eq:proof_l2_stab_pt_I}--\eqref{eq:proof_l2_stab_pt_IIV} yields \eqref{eq:l2_stab_pt}.
\end{proof}
\section{Stability and a priori error analysis}
\label{sec:analysis}
This section is devoted to the theoretical analysis of the
cut finite element method~\eqref{eq:biot-total-press-weak-disc-ghost} for the Biot problem.
We start by collecting the necessary stability and coercivity
estimates for the bilinear forms involved, before we turn to the
discussion of a modified inf-sup condition for the displacement-total-pressure
coupling bilinear form \(b_{1,h}(\cdot,\cdot)\) when mixed boundary conditions are considered.
Afterwards, we state and prove the main inf-sup condition for the total
bilinear form $A_h(\cdot,\cdot)$ given in~\eqref{eq:total_bilinear_form}.
Here, the main objective is to carry over the parameter robustness of
the total pressure formulation to the cut finite element setting,
ensuring that the resulting unfitted method is also geometrically robust.
We conclude the section by deriving optimal a priori error estimates
for the proposed cut finite element method.

\subsection{Stability estimates}
\begin{lemma}  \label{lemma:bilinear_bounds}
Assuming that $\gamma_u$ and $\gamma_p$ are sufficiently large, it holds that
\begin{align}
  A_{1,h}(\bv_h, \bv_h)  &\geqslant  c_{A_1} \tnormv{\bv_h}^2, &  \forall \bu_h,\bv_h \in \bV_h, \label{A1-coercive} \\
  A_{3,h}(q_{F,h}, q_{F,h})  &\geqslant  c_{A_3} \tnormf{q_{F,h}}^2, &  \forall p_{F,h},q_{F,h} \in Q_{F,h}, \label{A3-coercive}\\
  A_{1,h}(\bu_h, \bv_h) &\leqslant C_{A_1} \tnormv{\bu_h}\tnormv{\bv_h}, &  \forall \bu_h,\bv_h \in \bV_h, \label{A1-bounded}  \\
      a_{1,h}(\bu,\bv_{h}) &\lesssim  \normvast{\bu} \tnormv{\bv_{h}}, \ &\forall (\bu,\bv_h) \in [H^2(\Omega)]^n\times \bV_h, \label{a1-bounded}\\
    a_{2,h}(q_{T},q_{T,h}) &\lesssim  \lambda^{-1} \mu \|q_{T}\|_\Omega \tnormt{q_{T,h}}, \ &\forall (q_{T},q_{T,h}) \in L^2(\Omega)\times  Q_{T,h},\\
    a_{3,h}(q_{F},q_{F,h}) &\lesssim  \normfast{q_{F}} \tnormf{q_{F,h}}, \ &\forall (q_{F},q_{F,h}) \in H^2(\Omega)\times  Q_{F,h}, \label{a3-bounded}\\
    b_{1,h}(\bv, q_{T}) &\lesssim  \normv{\bv}\normtast{q_{T}}, &\forall (\bv,q_{T} ) \in [H^1(\Omega)]^n\times H^1(\Omega), \label{b1-bounded}\\
    b_{2,h} (q_{F}, q_{T}) &\lesssim \lambda^{-1} \norm{q_{F}}_{\Omega} \norm{q_{T}}_{\Omega}, \ &\forall  (q_{F}, q_{T}) \in L^2(\Omega)\times L^2(\Omega),\\
    L_1(\bv_h) &\lesssim \big(\mu^{-\onehalf} \|\bf\|_{\Omega}
    +\mu^{\onehalf} \norm{h^{-\onehalf} \bu_D}_{\Gamma_d}
    + \norm{ \bsigma_N }_{\tilde H^{-\onehalf}(\Gamma_s)} \big)\tnormv{\bv_h}, &  \forall \bv_h \in \bV_h, \label{L1-bound}\\
    L_2(q_{T,h}) &\lesssim \mu^{\onehalf} \norm{h^{-\onehalf} \bu_D}_{\Gamma_d} \tnormt{q_{T,h}}, & \forall q_{T,h} \in Q_{T,h}, \label{L2-bound}\\
    L_3(q_{F,h}) &\lesssim 
    \big(\lambda^{\onehalf}\|g\|_{\Omega}  
     + K^{\onehalf} \norm{h^{-\onehalf} p_{F,D}}_{\Gamma_s} + \norm{ g_N }_{\tilde H^{-\onehalf}(\Gamma_d)}\big)\tnormf{q_{F,h}}, & \forall q_{F,h} \in Q_{F,h}. \label{L3-bound}
\end{align}
\end{lemma}
\begin{proof}
The bilinear form \(a_{3,h}\) is associated with an elliptic problem with Nitsche's method, and \(a_{1,h}\) and \(b_{1,h}\) are associated with the Stokes problem with Nitsche's method, and similarly with the stabilized forms. Proof of coercivity and boundedness follows from the standard arguments for these bilinear forms, see e.g. \cite{Burman2010,Burman-Hansbo-cutfem-Nitsche,Massing-etal-cutfem-Stokes,MassingSchottWall2017}, and similarly for the boundness of \(a_{2,h}\), \(b_{2,h}\), $L_1$, $L_2$, and $L_3$. We emphasize a few key points in the argument. Bounds \eqref{A1-coercive}--\eqref{A1-bounded}, \eqref{L1-bound}, and \eqref{L3-bound} involve the stabilized norms $\tnormv{\cdot}$ and $\tnormf{\cdot}$. Special care needs to be taken to control the terms in $a_{1,h}$ and $a^1_{3,h}$ involving first derivatives on $\Gamma$. For example, we have
\begin{equation*}
  (\mu\varepsilon(\bu_h)\cdot \bn, \bv_h)_{\Gamma_d} \leqslant \mu^{\onehalf}h^{\onehalf}\|\varepsilon(\bu_h)\cdot \bn\|_{\Gamma_d}\mu^{\onehalf}h^{-\onehalf}\|\bv_h\|_{\Gamma_d} \lesssim \mu^{\onehalf}
  \|\varepsilon(\bu_h)\|_{\cT_h}\normv{\bv_h} \lesssim \tnormv{\bu_h}\tnormv{\bv_h},
\end{equation*}
where in the second inequality we used the trace inequality \eqref{eq:trace-ineq-2} and the inverse estimate \eqref{eq:inverse-estimates-I}, and in the last inequality we used Assumption~\textbf{A1}. The terms $(\bu_h, \mu\varepsilon(\bv_h)\cdot \bn)_{\Gamma_d}$ in $a_{1,h}$, $(K \nabla p_{F,h} \cdot \bn, q_{F,h})_{\Gamma_s}$ and $(p_{F,h}, K \nabla q_{F,h} \cdot \bn )_{\Gamma_s}$ in $a_{3,h}^1$, $ (\bu_D, \mu\varepsilon(\bv_h)\cdot \bn)_{\Gamma_d}$ in $L_1$, and $(p_{F,D}, K \nabla q_{F,h} \cdot \bn )_{\Gamma_s}$ in $L_3$ are handled similarly.

Furthermore, in bounds \eqref{a1-bounded}, \eqref{a3-bounded}, and \eqref{b1-bounded}, one of the arguments is bounded in the norms $\normvast{\cdot}$,  $\normfast{\cdot}$ or $\normtast{\cdot}$, which are enhanced with suitable boundary terms, allowing to control the boundary terms in the bilinear forms. 
\end{proof}

Next, we recall an
assumption from \cite{GuzmanOlshanskii2016}, where we make use of the function space
$ H^1_0(\cT_h^i) = \{ v \in H^1(\cT_h^i) : v = 0 \text{ on } \partial \cT_h^i \} $.

\begin{assumption}
  \label{Assumption_GO}
  \cite{GuzmanOlshanskii2016} Assume that there exists \(\tilde\beta > 0\) independent of \(h\), depending only on polynomial degree of finite element spaces and the shape regularity of \( \cT_{h} \) such that 
  \begin{align*}
    \tilde\beta \|  h \nabla q  \| _{\cT_h^i} 
    \leqslant 
    \sup_{\bv \in \bV_{h}^{i}} \frac{(\nabla \cdot \bv,q)_{\cT_h^i}}{\| \bv  \| _{1,\cT_h^i}}, &  \quad \forall q \in Q_{T,h},
  \end{align*}
  where \( \bV_{h}^{i} = \bV_{h}\cap[H^1_0(\cT_h^i)]^n \).
\end{assumption}
 
With this assumption, we can prove a modified inf-sup condition for \(b_{1,h}\),
where the defect term can be quantified in terms of the seminorm induced by the ghost penalty for the total pressure. But first we need to ensure that the corrected Scott-Zhang interpolant \(\pi^c_{h}\) defined in \eqref{eq:modiefied_interpolant} is stable in the ghost-penalty enhanced norm \(\tnormv{\cdot}\).

\begin{lemma} For any \(v \in H_{\Gamma_d,0}^1(\Omega)\) it holds that
  \begin{align}\label{interpolation_stability}
    \tnormv{\pi^c_{h} \bv }  &\lesssim \mu^{\onehalf} \| \bv \|_{1,\Omega}.  
  \end{align}
\end{lemma}

\begin{proof}
From the definition of $\tnormv{\cdot}$ (cf. \eqref{defn-norm-Vh} and \eqref{Vh-norm}), we have
\begin{align}\label{pihc-stab-0}
    \tnormv{\pi^c_{h} \bv } \leqslant \mu^{\onehalf} \norm{\varepsilon(\pi^c_{h} \bv)}_{\Omega}
    + \gamma_u^{\onehalf} \mu^{\onehalf} \|h^{-\onehalf} \pi^c_{h} \bv\|_{\Gamma_d}
    + |\pi^c_{h} \bv|_{g_{1,h}}.
\end{align}
For the first term, using the stability of \(\pi^c_h\) \eqref{eq:stab-pi-c}, we have
\begin{equation}\label{pihc-stab-1}
\mu^{\onehalf} \norm{\varepsilon(\pi^c_{h} \bv)}_{\Omega} \lesssim \mu^{\onehalf} \| \bv \|_{1,\Omega}.
\end{equation}
For the second term in \eqref{pihc-stab-0}, since we have \(v \in H_{\Gamma_d,0}^1(\Omega)\) we can write 
\begin{align}\label{pihc-stab-2}
\gamma_u^{\onehalf} \mu^{\onehalf}\|h^{-\onehalf} \pi^c_{h} \bv\|_{\Gamma_d} = \gamma_u^{\onehalf} \mu^{\onehalf}\|h^{-\onehalf} (\pi^c_{h} \bv - \bv)\|_{\Gamma_d} \lesssim \mu^{\onehalf}\| \bv \|_{1,\Omega},
\end{align}
where the last inequality comes from \eqref{interpolation_stab2}. For the third term in \eqref{pihc-stab-0} we use the triangle inequality, Assumptions \textbf{A4} and \textbf{A2} for the ghost penalty, and result \eqref{interpolation_stab2}:
\begin{align}\label{pihc-stab-3}
|\pi^c_{h} \bv|_{g_{1,h}} \leqslant |\pi^c_{h} \bv - \pi_{h} \bv|_{g_{1,h}} + |\pi_{h} \bv|_{g_{1,h}} 
  \lesssim h^{-1} \mu^{\onehalf} \norm{ \pi^c_{h} \bv - \pi_{h} \bv}_{\cT_h}
+ \mu^{\onehalf} \norm{\bv}_{1,\Omega}
 \lesssim \mu^{\onehalf} \| \bv \|_{1,\Omega}. 
\end{align}
Bound \eqref{interpolation_stability} follows by combining \eqref{pihc-stab-0}--\eqref{pihc-stab-3}.
\end{proof}

With these preparations, we can now prove the modified inf-sup condition for \(b_{1,h}\).
\begin{lemma}
Suppose that Assumption \ref{Assumption_GO} holds. There exist constants $\beta_1 > 0$ and $\beta_2 > 0$ such that 
\begin{equation}\label{inf-sup}
\sup_{\bv_h \in \bV_h} \frac{b_{1,h}(\bv_h, q_{T,h})}{\tnormv{\bv_h}} \geqslant \beta_1  \|q_{T,h} \|_{Q_{T,h}} - \beta_2  |q_{T,h}|_{g_{2,h}}, \quad \forall q_{T,h} \in Q_{T,h}.
\end{equation}
\end{lemma}
\begin{proof}
  Let \(q_{T,h} \in Q_{T,h} \) be given. There exists \(\bv^q \in \bH^1_{\Gamma_d,0} \) such that $- \nabla\cdot\bv^q = q_{T,h}$ in $\Omega$, and \(\beta  \|\bv^q\|_1 \leqslant\|q_{T, h}\|_{\Omega} \) for a constant \( \beta > 0 \). Let \(\bv^q_{h} = \pi^c_{h} \bv^q \) where  \(\pi^c_{h} \) is the modified Scott-Zhang interpolant defined in \eqref{eq:modiefied_interpolant}. Recalling the definition of \( b_{1,h}\) \eqref{defn-b1}, an integration by parts and the fact that \( \bv^q = 0\) on \(\Gamma_d\) yield
  \begin{align}
    b_{1,h}(\bv^q_{h},q_{T,h}) 
    &=  -(\nabla \cdot (\bv^q_h-\bv^q),q_{T,h})_{\Omega}  + ((\bv^q_h - \bv^q)  \cdot \bn, q_{T,h})_{\Gamma_d} - (\nabla \cdot  \bv^q,q_{T,h})_{\Omega} \nonumber \\
    &=  ((\bv^q_h-\bv^q),\nabla q_{T,h})_{\Omega}  - ((\bv^q_h - \bv^q) \cdot \bn, q_{T,h})_{\Gamma_s} + \| q_{T,h}  \|^2_{\Omega}  \label{eq:line_in_infsup_proof}
    \\
    &= \text{I} + \text{II} + \| q_{T,h}  \|_{\Omega}^2 
    \geqslant \text{I} + \text{II} + \| q_{T,h}  \|_\Omega \, \beta \, \|\bv^q\|_{1,\Omega}.
    \label{eq:bh_infsup_proof_step_1}
  \end{align}
  Combining the Cauchy-Schwarz inequality with the interpolation estimate~\eqref{interp-T} shows that the first term $\text{I}$ can be bounded as
  \begin{align*}
   \abs{\text{I}}
    &\lesssim \| h^{-1} ( \bv^q_h - \bv^q)  \| _{\Omega} \| h \nabla q_{T,h}  \| _{\Omega}
    \lesssim  \| \bv^q  \| _{1,\Omega} \| h \nabla q_{T,h}\|_{\Omega}.  
  \end{align*}

  Turning to the second term \(\text{II}\), 
  we utilize
  the orthogonality property of the modified Scott-Zhang interpolant \eqref{eq:normal_ort_prop}
  by inserting the patch-wise constant pressure
  \( \bar{q}_{i} = \frac{1}{|\cP_i|} (q_{T,h}, 1)_{\cP_i}  \),
  leading to
  \begin{align*}
    \abs{\text{II}} &=  \left| \sum_{i=1}^N ((\bv^q_h - \bv^q) \cdot \bn, q_{T,h} - \bar{q}_{i} )_{\Gamma_s \cap \cP_i} \right|\\
    &\lesssim \sum_{i=1}^{N} {\| h^{-\onehalf}(\bv^q_h - \bv^q) \cdot  \bn \|_{\Gamma_s \cap \cP_i} } \| h^{\onehalf} ( q_{T,h} - \bar{q}_{i})\|_{\Gamma_s \cap \cP_i}  \\
    &\lesssim    {\| \bv^q\|_{1,\Omega}}   \| h \nabla  q_{T,h} \|_{\cT_h}, 
\end{align*}
where in the last line we used \eqref{interpolation_stab2}, the trace inequality \eqref{eq:trace-ineq-2}, and a local Poincar\'e--Friedrichs inequality.
By inserting the bounds for $|\text{I}|$ and $|\text{II}|$
into \eqref{eq:bh_infsup_proof_step_1}, dividing by $\|\bv^q\|_{1,\Omega}$ and 
finally applying the stability of the interpolation operator $\pi_h^c$ \eqref{interpolation_stability},
we deduce that
  \begin{align*}
    \frac{b_{1,h}(\bv^q_{h},q_{T,h})}{\tnormv{\bv^q_h} } \gtrsim \frac{b_{1,h}(\bv^q_{h},q_{T,h})}{\mu^{\onehalf}\| \bv^q  \| _{1, \Omega} }
    &\gtrsim  \mu^{-\onehalf}\|  q_{T,h}  \|_{\Omega}
    - \mu^{-\onehalf} \left(\| h \nabla q_{T,h}\|_{\Omega}
    + \| h \nabla  q_{T,h} \|_{\cT_h} \right), 
  \end{align*}
where we assumed that \( b_{1,h}(\bv^q_{h},q_{T,h}) \geqslant 0 \) by choosing the sign of \( \bv^q_{h} \) accordingly.
Consequently, we arrive at, for some $\tilde\beta_1 > 0$ and $\tilde\beta_2 > 0$,
  \begin{align}
    \sup_{\bv_h \in \bV_h} \frac{b_{1,h}(\bv_h, q_{T,h})}{\tnormv{\bv_h}} \geqslant \tilde{\beta}_1  \normt{q_{T,h}} - \tilde{\beta}_2 \mu^{-\onehalf} \| h \nabla  q_{T,h} \|_{\cT_h} \label{eq:proof_sup1}. 
  \end{align}
  Next, thanks to~\eqref{eq:gp_extendpt} and Assumption \ref{Assumption_GO}, there exists \(\bv_{h} \in V_{h}\) with \(\supp{\bv_{h}} \subset \cT_{h}^{i} \), such that 
  \begin{align*}
    \mu^{-\onehalf} \| h \nabla q_{T,h} \| _{\cT_h} 
    &\lesssim
    \mu^{-\onehalf} \| h \nabla q_{T,h} \| _{\Omega_{h}^i}  
    + |q_{T,h}|_{g_{2,h}}
    \lesssim
    \sup_{\bv_h \in \bV_{h}^{i}} \frac{(\nabla \cdot \bv_h,q)_{\cT_h^i}}{\mu^{\onehalf} \| \bv_h  \| _{1,\cT_h^i}}
    + |q_{T,h}|_{g_{2,h}}
    \lesssim 
    \sup_{\bv_h \in \bV_{h}^{i}}   \frac{b_{1,h}(\bv_{h},q_{T,h})}{ \tnormv{ \bv_{h}}} +  |q_{T,h}|_{g_{2,h}},
  \end{align*}
  where in the last step we used that for \( \bv_h \in V_h^i\) we have
  \( b_{1,h}(\bv_{h},q_{T,h}) = -(\nabla \cdot \bv_h, q_{T,h})_{\cT_h^i} \) and
  the norm equivalence \( \mu^{\onehalf}\| \bv_{h}  \|_{1,\cT_h^i}  \simeq   \tnormv{ \bv_{h}}\) for $\bv_h \in \bV_h^i$.
  In other words, it holds that
  \begin{align}
    \sup_{\bv_h \in \bV_h} \frac{b_{1,h}(\bv_h, q_{T,h})}{\tnormv{\bv_h}} \geqslant \tilde{\beta}_3    \| h \nabla  q_{T,h} \|_{\cT_h} - \tilde{\beta}_4  |q_{T,h}|_{g_{2,h}} \label{eq:proof_sup2},
  \end{align}
  for some positive constants \(\tilde{\beta}_3\) and \(\tilde{\beta}_4\).
  Now the inf-sup condition~\eqref{inf-sup} follows easily from combining~\eqref{eq:proof_sup1} and~\eqref{eq:proof_sup2}.
\end{proof}

Before we turn to the proof of the main inf-sup condition for the total bilinear form \(A_h(\cdot,\cdot)\), we want to make the following important remarks.

\begin{remark}[General inf-sup stable elements]
It is shown in \cite{GuzmanOlshanskii2016} that
Assumption~\ref{Assumption_GO} holds for various mixed finite element
spaces with continuous pressure spaces, including generalized Taylor-Hood elements, 
the Bercovier-Pironneau element, and the the Mini-element.
\end{remark}
\begin{remark}[Discontinuous pressure elements]
As shown in \cite{GuzmanOlshanskii2016}, inf-sup stable elements with
discontinuous pressure spaces can also be used, provided that the pressure norm
in Assumption~\ref{Assumption_GO} is modified to include
the pressure jump term $h^{\onehalf} \| \jump{p_h}\|_{\cF_h^i(\cT_h^i)}$ 
across interior faces associated with the interior mesh $\cT_h^i$.
The integration by parts step in the proof of the intermediate
modified inf-sup condition~\eqref{eq:proof_sup1} 
will then lead to an additional term of the form
$-\widetilde{\beta}_5 h^{\onehalf} \| \jump{p_h}\|_{\cF_h^i\cap{\Omega}}$ on the right-hand side
of \eqref{eq:proof_sup1}.
Consequently, the ghost penalty stabilization for the total pressure must then be designed
to control this additional term so that
\begin{align*}
  \|h\nabla p_{T,h} \|_{\cT_h} + h^{\onehalf} \| \jump{p_h}\|_{\cF_h^i\cap{\Omega}} \lesssim \|h\nabla p_{T,h} \|_{\cT_h^i} + \mu^{\onehalf} |p_{T,h}|_{g_{2,h}} \lesssim \| p_{T,h} \|_{\cT_h}, 
  \quad \forall p_{T,h} \in Q_{T,h}, 
\end{align*}
holds instead of only \eqref{eq:gp_extendpt}, where $\|h\nabla p_{T,h} \|_{\cT_h}$ is understood as a broken norm. This is automatically satisfied for 
the local projection and volume based ghost penalties, as well as for the face-based ghost penalties
since we included index $j=0$ in the summation.
As a result, the presented analysis also covers inf-sup stable elements with discontinuous pressure spaces,
such as generalized conforming Crouzeix-Raviart elements, 
the $\bbP_{\mathrm{c}}^{k+d}$-$\bbP_{\mathrm{dc}}^k$ elements in $d=2,3$ dimensions, and the Bernardi-Raugel element.
See \cite[Section 6]{GuzmanOlshanskii2016} for details.
\end{remark}
\begin{remark}
  The proof of the modified inf-sup condition \eqref{inf-sup} for \(b_{1,h}(\cdot,\cdot)\)
  varies significantly from the one presented in \cite{GuzmanOlshanskii2016}.
  First, by exploiting the extension property of the total pressure ghost penalty with
  respect to the $h$-weighted gradient norm, we avoid the discussion of handling pressure averages
  with respect to different domains (which is not required in our setting anyhow because of the mixed boundary conditions).
  On the other hand, our analysis is complicated by the presence of $\Gamma_s$ / mixed boundary conditions, 
  as we now need to handle the boundary terms on $\Gamma_s$ appearing after integration by parts in \eqref{eq:line_in_infsup_proof}.
  To translate these terms into a bound involving the $h$-weighted gradient norm of the pressure,
  we make use of the orthogonality property of the modified Scott-Zhang interpolant \eqref{eq:normal_ort_prop}.
\end{remark}

\begin{theorem} \label{thm:inf-sup_A}
  For all \((\bu_h,p_{T,h},p_{F,h})  \in \bV_h \times Q_{T,h} \times Q_{F,h}  \) it holds that 
  \begin{align} \label{eq:inf-sup-big-A}
  \sup_{(\bv_h,q_{T,h},q_{F,h})  \in \bV_h \times Q_{T,h} \times Q_{F,h}} \frac{A_{h}\bigl((\bu_h,p_{T,h},p_{F,h}),(\bv_h,q_{T,h},q_{F,h})\bigr)}{\tn(\bv_h,q_{T,h},q_{F,h})\tn_h} \gtrsim   \tn{(\bu_{h},p_{T,h},p_{F,h})} \tn_h.
  \end{align}
\end{theorem}
\begin{proof}
  For a given \(  (\bu_h,p_{T,h},p_{F,h})\) we want to construct functions \((\bv_h, q_{T,h}, q_{F,h})\) satisfying 
  \begin{align*}
    A_{h}\bigl((\bu_h,p_{T,h},p_{F,h}),(\bv_h,q_{T,h},q_{F,h}\bigr)) \gtrsim 
    \tn{(\bu_{h},p_{T,h},p_{F,h})} \tn_h
    \tn(\bv_h,q_{T,h},q_{F,h})\tn_h.
  \end{align*}
  \textbf{Step 1.}
First, we let \( (\bv_h, q_{T,h}, q_{F,h}) =(\bu_h,-p_{T,h}, -p_{F,h}) \) in the definition of $A_h$ \eqref{eq:total_bilinear_form}. Using
the coercivity properties \eqref{A1-coercive} and \eqref{A3-coercive} and
the identity (cf. \eqref{defn-b2}, \eqref{defn-a2}, and \eqref{defn-a3-2})
\begin{align*}
  a_{2,h}(p_{T,h}, p_{T,h}) - 2 b_{2,h}(p_{F, h}, p_{T,h}) + a^2_{3,h}(p_{F,h}, p_{F,h}) =  \frac{1}{\lambda}\|p_{T,h } -p_{F,h} \|_{\Omega}^2 + \frac{1}{\lambda}\|p_{F,h }\|_{\Omega}^2,
\end{align*}
we obtain
\begin{align}
  &A_{h}((\bu_h,p_{T,h},p_{F,h}),(\bu_h,-p_{T,h},-p_{F,h})) \nonumber \\ 
  &= A_{1,h}(\bu_h, \bu_h) +  A_{2,h}(p_{T,h}, p_{T,h})- 2 b_{2,h}(p_{F, h}, p_{T,h}) +  A_{3,h}(p_{F,h}, p_{F,h}) \nonumber  \\
  &= A_{1,h}(\bu_h, \bu_h)  + |p_{T,h}|^2_{g_{2,h}}   + \frac{1}{\lambda}\|p_{T,h } -p_{F,h} \|_{\Omega}^2 +  \frac{1}{\lambda}\|p_{F,h }\|_{\Omega}^2  +   A^1_{3,h}(p_{F,h}, p_{F,h}) + |p_{F,h}|^2_{g^2_{3,h}} \nonumber \\
  &\geqslant c_{A_1} \tnormv{\bu_h}^2  + |p_{T,h}|^2_{g_{2,h}} + c_{A_3} \tnormf{p_{F,h}}^2 . 
  \label{eq:inf-sup-A-step1}
\end{align}
\textbf{Step 2.}
Next, we choose \( \bv_h = \bw_{h} \)
where \(\bw_{h}\) satisfies the modified inf-sup
condition~\eqref{inf-sup} for \(p_{T,h}\) and is rescaled such that
\(\tnormv{\bw_h} =  \normt{p_{T,h}}\). Utilizing \eqref{A1-bounded} and \eqref{inf-sup} leads to
\begin{align}
  A_{h}&\bigl((\bu_h,p_{T,h},p_{F,h}),(\bw_{h},0,0)\bigr)  
  = A_{1,h}(\bu_h, \bw_h) + b_{1,h}(\bw_h, p_{T,h}) \nonumber
  \\
  &\geqslant -C_{A_1}\tnormv{\bu_h}\tnormv{\bw_h}  + \beta_1  \normt{p_{T,h}}\tnormv{\bw_h} -  \beta_2  |p_{T,h}|_{g_{2,h}}\tnormv{\bw_h} \nonumber\\
  &= -C_{A_1} \tnormv{\bu_h} \normt{p_{T,h}}  + \beta_1    \normt{p_{T,h}}^2 -  \beta_2  |p_{T,h}|_{g_{2,h}}\normt{p_{T,h}} \nonumber
  \\
  &\geqslant 
  - \frac{C_{A_1}}{4\varepsilon}\tnormv{\bu_h}^2 
  + ( \beta_1   - C_{A_1} \varepsilon - \beta_2  \varepsilon )\normt{p_{T,h}}^2   
  -  \frac{\beta_2  }{4\varepsilon}|p_{T,h}|^2_{g_{2,h}} \nonumber
  \\
  &\geqslant - \frac{C_{A_1}(C_{A_1}+\beta_2  )}{2\beta_1  }\tnormv{\bu_h}^2 + \dfrac{\beta_1  }{2}\normt{p_{T,h}}^2
  -  \frac{\beta_2  (C_{A_1}+\beta_2 )}{2\beta_1  }|p_{T,h}|^2_{g_{2,h}},
  \label{eq:inf-sup-A-step2}
\end{align} 
where in the last 2 steps, an \( \varepsilon \)-scaled Young's inequality was used together with the choice \( \varepsilon = \frac{\beta_1   }{2(C_{A_1}+\beta_2  )} \). 

\noindent
\textbf{Step 3.}
Finally, combining the bounds~\eqref{eq:inf-sup-A-step1} and ~\eqref{eq:inf-sup-A-step2} and recalling the definition \eqref{defn-tri-norm} of
$\tn{(\bu_{h},p_{T,h},p_{F,h})} \tn_h$,
we see that for \(\delta >0 \)
\begin{align*}
  A_{h}((\bu_h,p_{T,h},p_{F,h}),(\bu_h+ \delta \bw_h,-p_{T,h},-p_{F,h}))
  &\geqslant 
  \left( c_{A_1} - \delta \frac{ C_{A_1}(C_{A_1} + \beta_2 )}{2\beta_1  }
  \right) \tnormv{\bu_h}^2  + \delta \frac{\beta_1  }{2}   \normt{p_{T,h}}^2  
  \\
  &\quad
  + \left(1-\delta \frac{\beta_2  (C_{A_1}+\beta_2  )}{2\beta_1}\right)|p_{T,h}|^2_{g_{2,h}} 
  + c_{A_{3}} \tnormf{p_{F,h}}^2.
\end{align*}
Choosing 
\( \delta < \min\left\{ \frac{\beta_1   c_{A_1}}{C_{A_1}(C_{A_1} + \beta_2  )}, \frac{\beta_1}{\beta_2   (C_{A_1} + \beta_2  )} \right\} \), 
we conclude that there is a constant \( c_A > 0 \) such that 
\begin{align*}
  &A_{h}\bigl((\bu_h,p_{T,h},p_{F,h}),(\bu_h+ \delta \bw_h,-p_{T,h},-p_{F,h})\bigr)
  \geqslant c_A \tn{(\bu_{h},p_{T,h},p_{F,h})} \tn_h^2 
  \\
  &\qquad\geqslant c_A (1+\delta)^{-1} \tn{(\bu_{h},p_{T,h},p_{F,h})} \tn_h \tn (\bu_h+ \delta \bw_h,-p_{T,h},-p_{F,h})\tn_h.
\end{align*}
where in the last step we used that
$
\tn (\bu_h+ \delta \bw_h,-p_{T,h},-p_{F,h})\tn_h \leqslant 
(1+\delta)\tn{(\bu_{h},p_{T,h},p_{F,h})} \tn_h.
$
\end{proof}

\begin{corollary}
  Method \eqref{eq:biot-total-press-weak-disc-ghost} has a unique solution satisfying
  \begin{equation} \label{stab-bound}
    \begin{split}
  \tnormv{\bu_h} + \tnormt{p_{T,h}} + \tnormf{p_{F,h}} &\lesssim  \mu^{-\onehalf} \|\bf\|_{\Omega} +\mu^{\onehalf} \norm{h^{-\onehalf} \bu_D}_{\Gamma_d}
   + \norm{ \bsigma_N }_{\tilde H^{-\onehalf}(\Gamma_s)} \\
   &\quad + \lambda^{\onehalf}\|g\|_{\Omega}  
     + K^{\onehalf} \norm{h^{-\onehalf} p_{F,D}}_{\Gamma_s} + \norm{ g_N }_{\tilde H^{-\onehalf}(\Gamma_d)} . 
    \end{split}
    \end{equation}
  \end{corollary}
\begin{proof}
We first establish \eqref{stab-bound}. Using the reformulation \eqref{method-v2}, for a solution \( (\bu_h,p_{T,h},p_{F,h}) \) of \eqref{eq:biot-total-press-weak-disc-ghost}, Theorem~\ref{thm:inf-sup_A} implies
\begin{align*}
    \tn{(\bu_{h},p_{T,h},p_{F,h})} \tn_h &\lesssim 
      \sup_{(\bv_h,q_{T,h},q_{F,h})  \in \bV_h \times Q_{T,h} \times Q_{F,h}} \frac{A_{h}\bigl((\bu_h,p_{T,h},p_{F,h}),(\bv_h,q_{T,h},q_{F,h})\bigr)}{\tn(\bv_h,q_{T,h},q_{F,h})\tn_h}\\
      &=   \sup_{(\bv_h,q_{T,h},q_{F,h})  \in \bV_h \times Q_{T,h} \times Q_{F,h}} \frac{L_1(\bv_h)+L_2(q_{T,h})+L_3(q_{F,h})}{\tn(\bv_h,q_{T,h},q_{F,h})\tn_h},
\end{align*}
which, together with the continuity bounds \eqref{L1-bound}--\eqref{L3-bound}, implies \eqref{stab-bound}. Existence and uniqueness of the solution follow from \eqref{stab-bound} by setting the data to zero and concluding that \( (\bu_h,p_{T,h},p_{F,h})  = (\mathbf{0},0,0)\).
\end{proof}

\subsection{Error estimates}
With the inf-sup condition for the total bilinear form $A_h(\cdot,\cdot)$ at hand, 
an a priori error estimate can now be derived combining a weak
Galerkin orthogonality with the approximation properties of the
interpolation operator and the weak consistency of the CutFEM stabilization terms.

\begin{lemma}[Weak Galerkin orthogonality] 
 Let \((\bu ,p_T ,p_{F}) \in H^r(\Omega) \times H^s(\Omega) \times H^t(\Omega)  \) be the solution of the Biot problem \eqref{eq:biot-total-press-weak-disc}, and let \((\bu_h,p_{T,h},p_{F,h}) \in \bV_h \times Q_{T,h} \times Q_{F,h}  \) be solution of the discrete stabilized formulation \eqref{eq:biot-total-press-weak-disc-ghost}. Then 
  \begin{align} \label{eq:weak_galerkin}
    B_{h}((\bu - \bu_h,p_T - p_{T,h},p_{F} - p_{F,h}),(\bv_h,q_{T,h},q_{F,h}))  = G_h((\bu_h,p_{T,h},p_{F,h}),(\bv_h,q_{T,h},q_{F,h})),
  \end{align}
  for all \( (\bv_h,q_{T,h},q_{F,h}) \in \bV_h \times Q_{T,h} \times Q_{F,h}  \). 
\end{lemma}

\begin{theorem}[A priori error estimate]\label{thm:error}
  Let \((\bu ,p_T ,p_{F}) \in H^r(\Omega) \times H^s(\Omega) \times H^t(\Omega) \) be the solution of the Biot problem \eqref{eq:biot-total-press-weak-disc}, and let \((\bu_h,p_{T,h},p_{F,h}) \in \bV_h \times Q_{T,h} \times Q_{F,h}  \) be the solution of the discrete stabilized formulation \eqref{eq:biot-total-press-weak-disc-ghost}. Then with \( \bar{k} + 1 = \min\{k+1,r,s\} \) and \( \bar{l}+1= \min\{l+1,t\}\)  the error \((\bu - \bu_h,p_T - p_{T,h},p_{F} - p_{F,h})\) satisfies 
  \begin{align*}
    \| (\bu - \bu_h,p_T - p_{T,h},p_{F} - p_{F,h})\|_{*} &\lesssim \mu^{\onehalf} h^{\bar{k}}  |\bu|_{\bar{k}+1,\Omega} + \left( \mu^{-\onehalf} + \lambda^{-\onehalf} + \lambda^{-1} \mu^{\onehalf} \right) h^{\bar{k}}  |p_T|_{\bar{k},\Omega} \\
     & \quad + \left( K^{\onehalf}  + \lambda^{-\onehalf} h+ \lambda^{-1}\mu^{\onehalf} h \right) h^{\bar{l}} |p_F|_{\bar{l}+1,\Omega}.
  \end{align*}
\end{theorem}

\begin{proof}
The proof follows the usual recipe.
We begin by decomposing the error into an interpolation error and a discrete error,
$$
(\bu - \bu_h,p_T - p_{T,h},p_{F} - p_{F,h}) = 
(\bu - \pi_{h} \bu ,p_T - \pi_{h} p_{T},p_{F} - \pi_{h} p_{F}) - 
( \bu_h - \pi_{h}\bu , p_{T,h}-\pi_{h} p_T ,p_{F,h}- \pi_{h} p_{F} ) =: E_{\pi} - E_h,
$$
resulting in 
\begin{align*}
 \| (\bu - \bu_h,p_T - p_{T,h},p_{F} - p_{F,h}) \|_{*}
    &\lesssim \|  E_{\pi} \|_{*} + \tn E_h  \tn_h.
  \end{align*}
  Starting with the $E_h$ term, 
  we take \( (\bv_h,q_{T,h},q_{F,h}) = W_h \) which satisfy the inf-sup estimate \eqref{eq:inf-sup-big-A} for \(A_h\) and scales such that 
  \( \tn W_h \tn_h = 1  \), 
  and utilize the weak Galerkin orthogonality \eqref{eq:weak_galerkin}
  to bound the discrete error by the interpolation error and the
  consistency error of the ghost penalty,
  \begin{align*} 
    \tn E_h  \tn_h 
    \lesssim A_{h}(E_h, W_h )
    = &B_{h}( E_{\pi}, W_h) - G_h( \pi_h U, W_h),
  \end{align*}
  where we use  \( \pi_h U = (\pi_h \bu, \pi_h p_T, \pi_h p_F) \).
  Assumption \textbf{A2} regarding weak consistency for the ghost penalties and the chosen scaling yields
  \begin{align*}
    |\pi_h U|_{G_h} \lesssim \mu^{\onehalf} h^{\bar{k}} |\bu|_{\bar{k}+1,\Omega} +  \mu^{-\onehalf} h^{\bar{k}}  |p_T|_{\bar{k},\Omega}+ \left( K^{\onehalf} h^{\bar{l}} + \lambda^{-\onehalf} h^{\bar{l}+1} \right) |p_F|_{\bar{l}+1,\Omega}.
\end{align*}
The bilinear form \(B_h\) can be bounded using the results in Lemma~\ref{lemma:bilinear_bounds}, the definitions of the enhanced norms \eqref{eq:normstar_1}--\eqref{eq:normstar_3} and \eqref{defn-norm-Vh}--\eqref{defn-norm-Qf}, and the inequality $\normtast{q_{T,h}} \lesssim \tnormt{q_{T,h}}$, which follows from the trace inequality~\eqref{eq:trace-ineq-2} and Lemma~\ref{lemma:l2-extension-pt}. We obtain
\begin{align*}
   B_{h}( E_h, W_h) 
  &\lesssim \normvast{\bu - \pi_{h}\bu } \tnormv{ \bv_h} 
  + \normv{\bv_h} \normtast{p_{T} - \pi_{h} p_T} 
 +  \normv{\bu - \pi_{h}\bu}\normtast{q_{T,h}} \\
  &\quad + \lambda^{-1} \mu \|p_{T} - \pi_{h}p_T\|_\Omega \tnormt{q_{T,h}} 
  + \lambda^{-1} \norm{ p_{F} - \pi_{h} p_F }_{\Omega}\norm{q_{T,h}}_\Omega\\
  &\quad+ \lambda^{-1}  \norm{ p_{T} - \pi_{h} p_T}_{\Omega} \norm{q_{F,h}}_\Omega  + \normfast{ p_{F} - \pi_{h} p_F } \tnormf{q_{F,h}} \\
  &\lesssim  \normvast{\bu - \pi_{h}\bu } \tnormv{ \bv_h}
  + \tnormv{\bv_h} \normtast{p_{T} - \pi_{h} p_T} 
  +  \normvast{\bu - \pi_{h}\bu}\tnormt{q_{T,h}} \nonumber \\
  &\quad + \lambda^{-1} \mu \|p_{T} - \pi_{h}p_T\|_\Omega \tnormt{q_{T,h}} 
   + \lambda^{-1}\mu^{\onehalf} \norm{ p_{F} - \pi_{h} p_F }_{\Omega}\tnormt{q_{T,h}} \nonumber \\
  &\quad+ \lambda^{-\onehalf}  \norm{ p_{T} - \pi_{h} p_T}_{\Omega} \tnormf{q_{F,h}} 
   + \normfast{ p_{F} - \pi_{h} p_F } \tnormf{q_{F,h}} \nonumber\\
  &\lesssim   \mu^{\onehalf} h^{\bar{k}}  |\bu|_{\bar{k}+1,\Omega}  + \left( \mu^{-\onehalf} + \lambda^{-\onehalf} + \lambda^{-1} \mu^{\onehalf}  \right) h^{\bar{k}}  |p_T|_{\bar{k},\Omega} 
 + \left( K^{\onehalf} h^{\bar{l}} + \lambda^{-\onehalf} h^{\bar{l}+1}+  \lambda^{-1}\mu^{\onehalf} h^{\bar{l}+1} \right)  |p_F|_{\bar{l}+1,\Omega},
\end{align*}
where the last line follows from the interpolation estimates \eqref{eq:energy_interpolation}. The assertion of the theorem follows by combining the above bounds.
\end{proof}

\section{Numerical results }
\label{sec:numerics}
In this section, we present numerical results to corroborate the theoretical findings. First, we present a convergence study in 2D, where we vary the parameters \( \lambda \) and \( K \) to show that the method is robust in these parameters. Next, we test the geometrical robustness of the method by varying the cut configurations. Afterwards, we present a 3D convergence study, and finally we present an example for solving Biot's equations on a brain geometry. In all experiments we use Taylor-Hood elements for the displacement and total pressure and second order elements for the fluid pressure, i.e. \( k = 2, l = 2 \). As ghost-penalty realizations we use the facet-based stabilizations \eqref{eq:gc_face_based}, with stabilization parameters \(\gamma_1 = 0.1, \gamma_2 = 0.01\), and Nitsche parameters \(\gamma_u = \gamma_p = 40.0 \).
The implementation is done in the Julia based open-source library Gridap~\cite{Badia2020,Verdugo2022}, with plugins GridapEmbedded, GridapPETSc, and STLCutters~\cite{badiaGeometrical2022}, and all experiments are run with Julia version 1.11.2. 
\begin{figure}
  \centering
  \includegraphics[trim = {6cm 25cm 7cm 13cm},clip=true, width=0.35\textwidth]{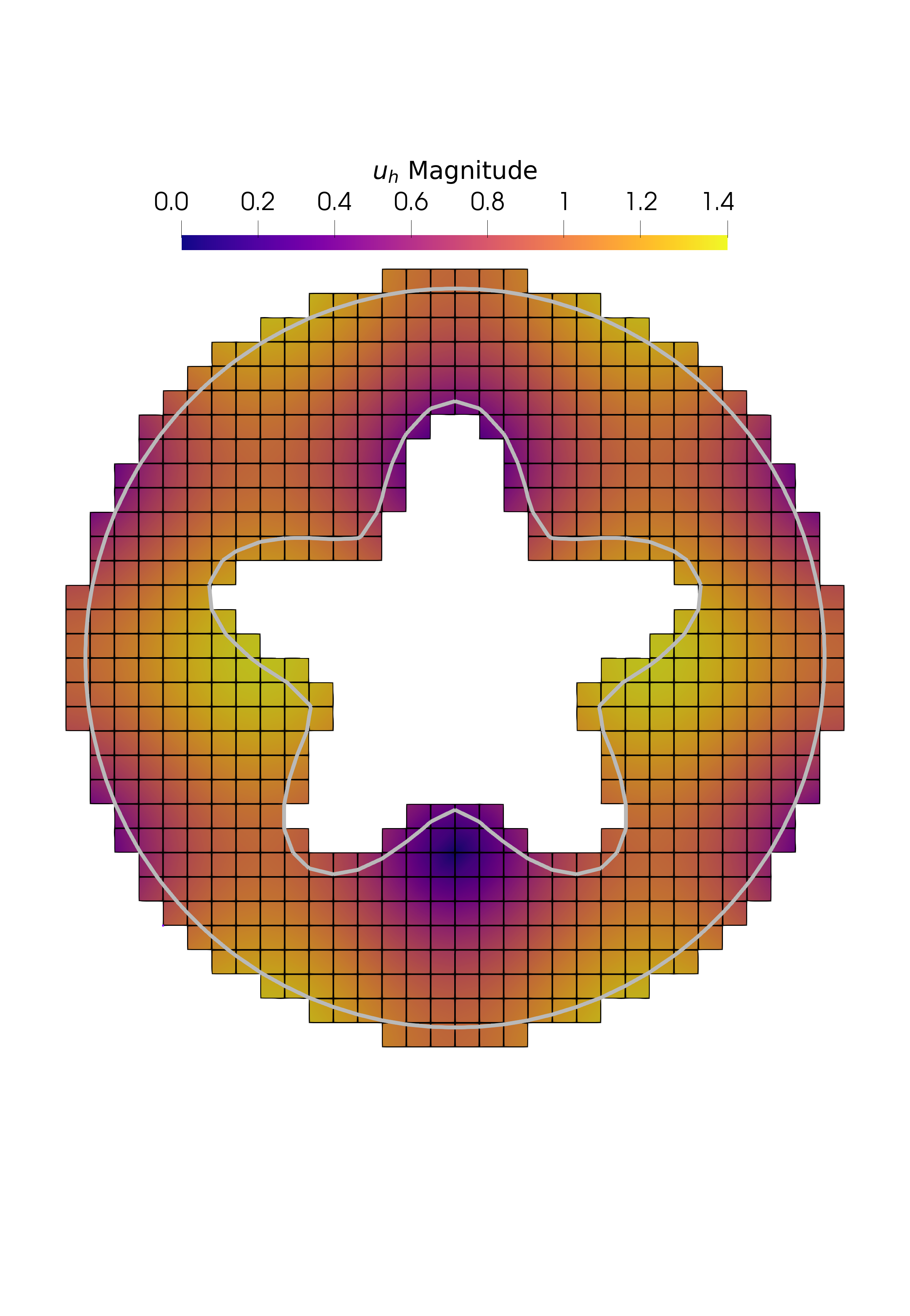}
  \includegraphics[trim = {6cm 25cm 7cm 13cm} ,clip=true,width=0.35\textwidth]{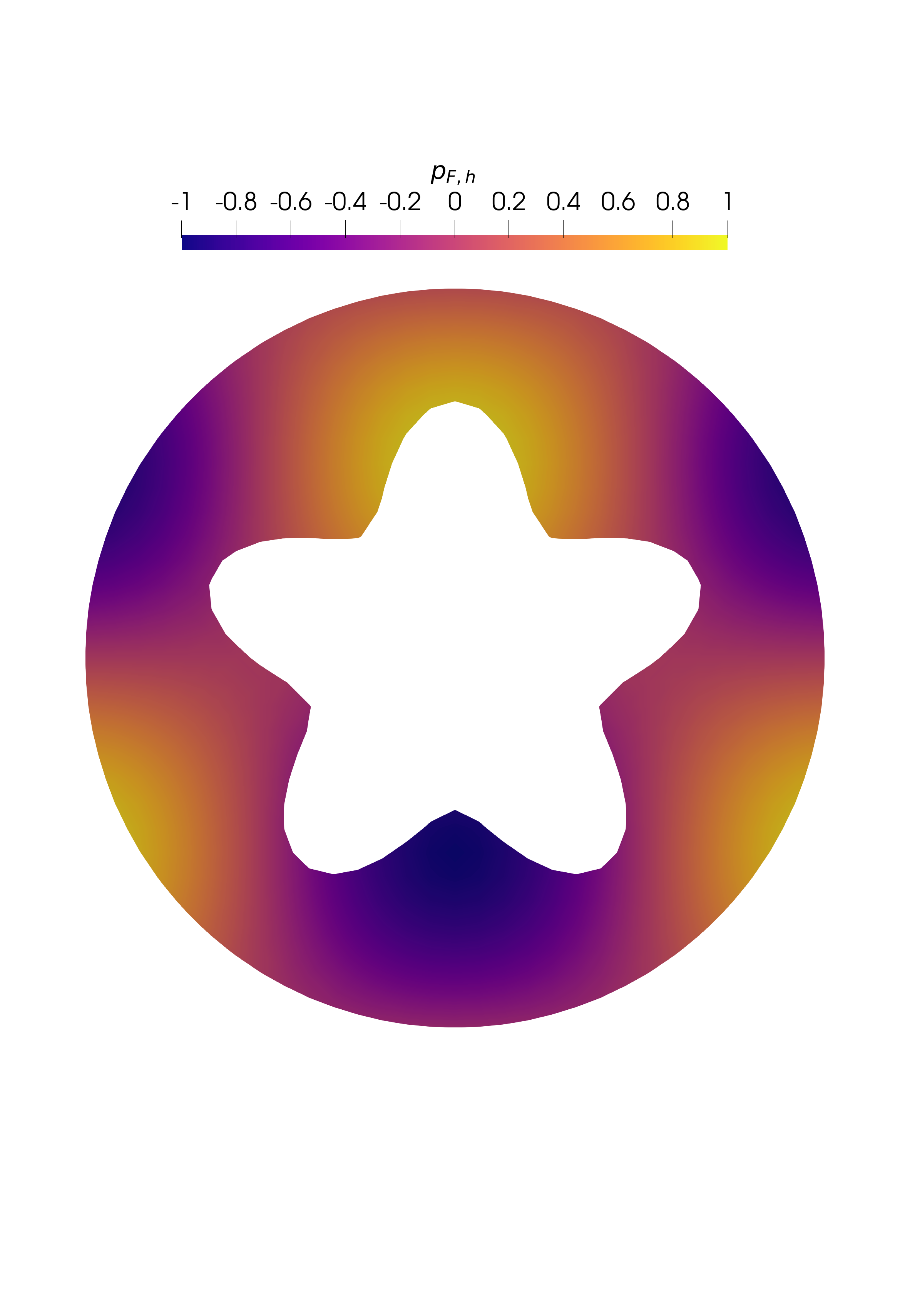}
  \caption{Convergence analysis 2D. The displacement solution \(\bu_h\) (magnitude) on the full computational domain and the fluid pressure solution \(p_{F,h}\) on the physical domain for \(N=32\) and \(\lambda =1\) and  \(K=1\).}
  \label{fig:solution}
\end{figure}

\subsection{Convergence analysis 2D}

We begin with a convergence analysis in 2D and define the background domain \( \widetilde{\Omega} = [-1,1]^2 \), the  circle 
\begin{align*} 
  \widetilde{\Omega}_1 = \{ (x,y) \in \mathbb{R}^2: x^2 + y^2 < 0.95^2 \},
\end{align*}
and a flower-like domain
\begin{align*}
  \widetilde{\Omega}_2 &= \{ (x,y) \in \mathbb{R}^2 : \phi(x,y) < 0 \}, \quad 
  \phi(x,y) =\sqrt{x^2 + y^2} -r_0 - r_1 \cos(5.0 \text{atan}_2 (y, x)),
\end{align*}
with \(  r_0 = 0.7, r_1= 0.18 \). The computational domain is defined as 
\begin{align*}
  \Omega = \widetilde{\Omega}_1 \setminus  \widetilde{\Omega}_2,  \quad \Gamma_d = \partial \widetilde{\Omega}_1 , \quad \Gamma_s = \partial \widetilde{\Omega}_2,
\end{align*}
i.e. a circle with a flower-shaped hole; see Figure~\ref{fig:solution} for an illustration.

\begin{figure}[p]
  \centering
  \begin{tabular}{cc}
      \begin{tabular}{c}
          \includegraphics[width=0.43\textwidth]{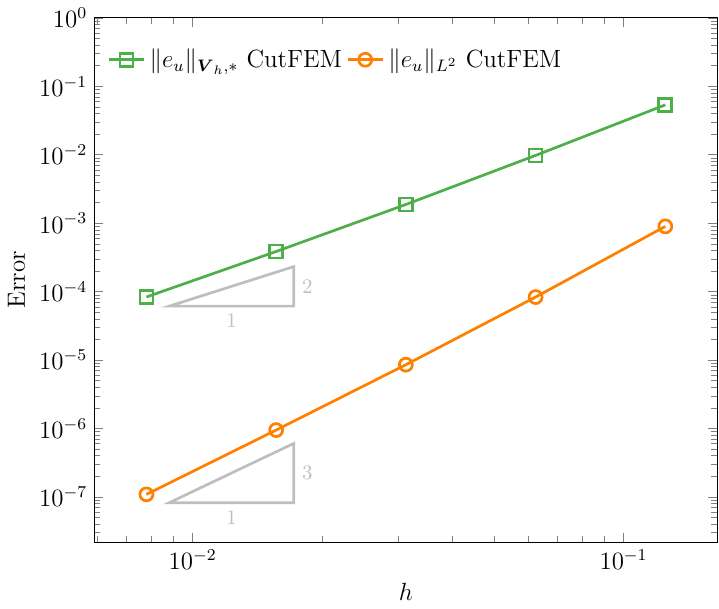} \\
          \includegraphics[width=0.43\textwidth]{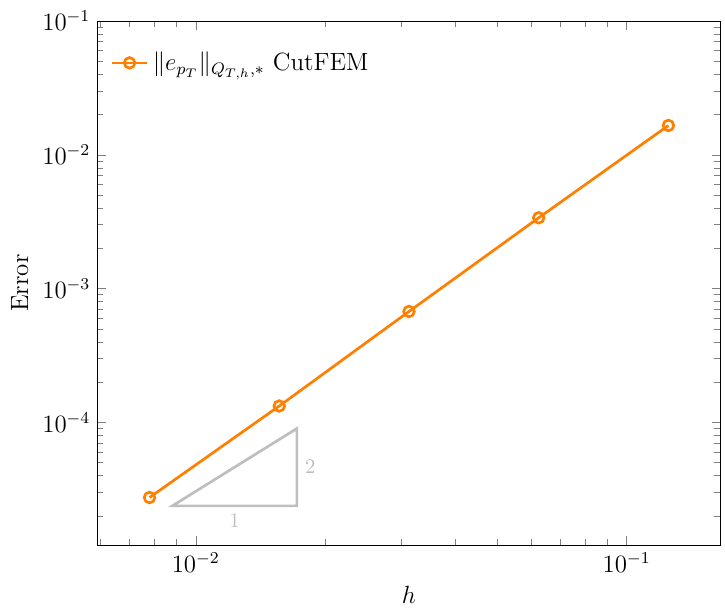} \\
          \includegraphics[width=0.43\textwidth]{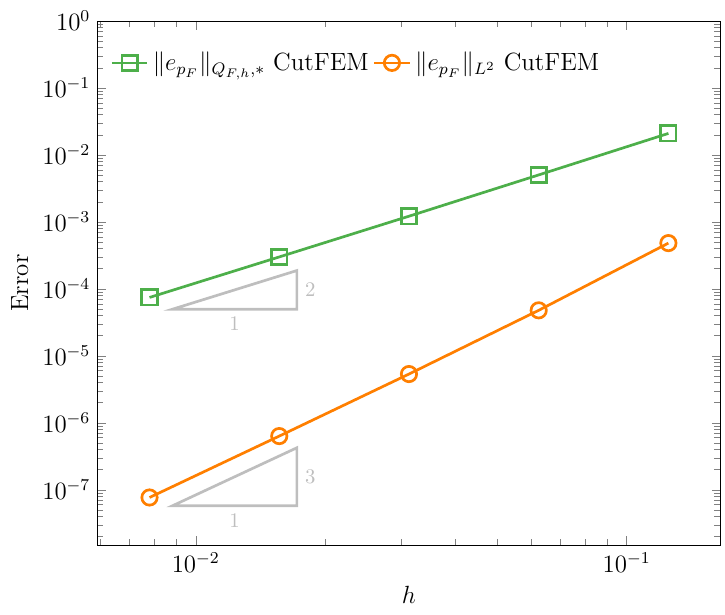} \\
      \end{tabular} &
      \begin{tabular}{c}
          \includegraphics[width=0.43\textwidth,trim=4 0 0 0,clip]{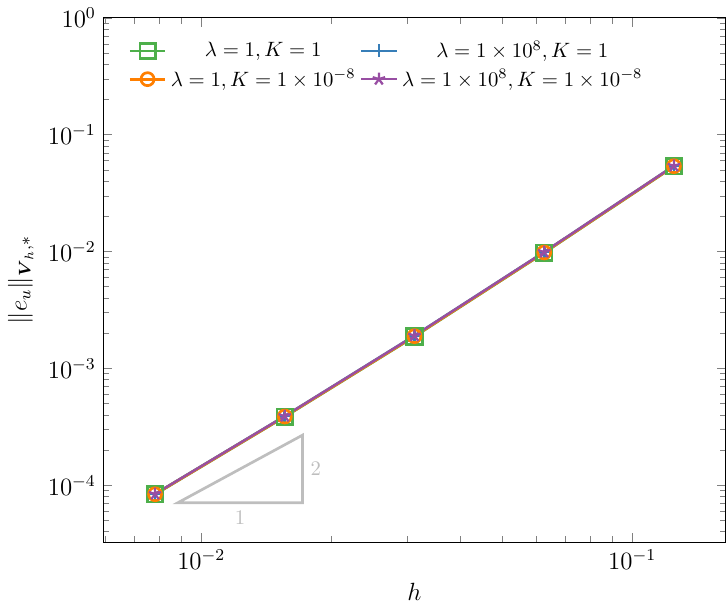} \\
          \includegraphics[width=0.43\textwidth,trim=4 0 0 0,clip]{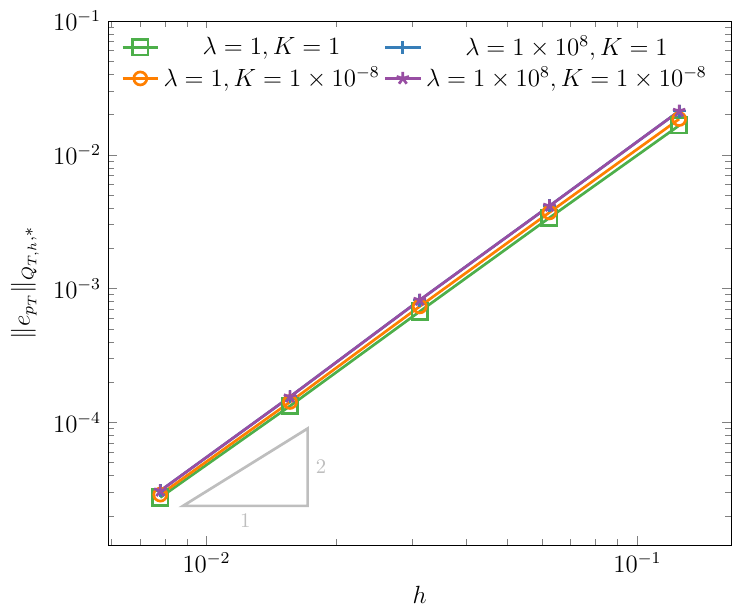} \\
          \includegraphics[width=0.43\textwidth,trim=4 0 0 0,clip]{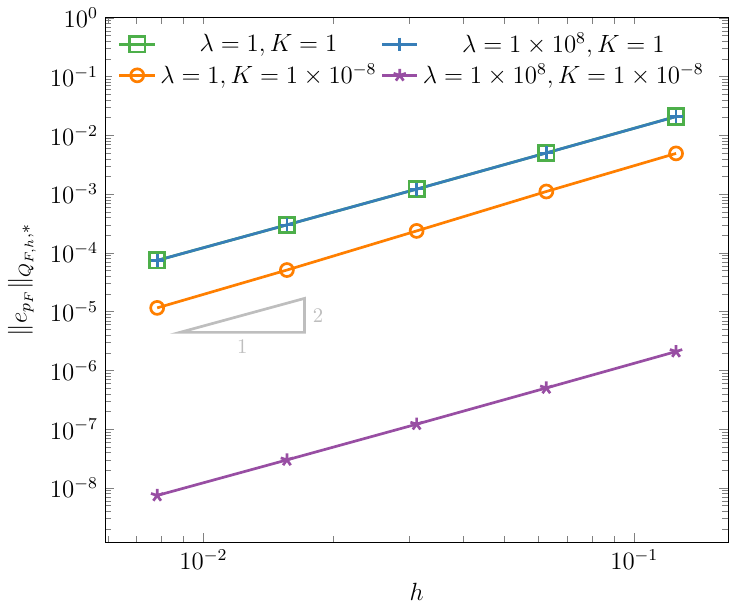} \\
      \end{tabular}
  \end{tabular}
  \caption{Convergence analysis 2D. Left: Errors measured in both the discrete norms \eqref{eq:normstar_1}-\eqref{eq:normstar_3} and \(L^2\)-norm for \(\lambda =1\) and  \(K=1\). Right: Errors measured in the discrete norms for different values of \(\lambda\) and \(K\).}
  \label{fig:combined-convergence}
\end{figure}

The problem is solved on a uniform mesh consisting of \(N^2\) uniform square elements, where the element size \(h_n\) is gradually decreased by refining the mesh, setting \(N = 2^{2+n}\), with \(n = 2, . . . , 6.\) For each refinement, the corresponding error \(E_n\) for the three unknowns is calculated in related norms,  
and the experimental order of convergence (EOC) is calculated via
\begin{align*}
  \text{EOC} = \frac{\log(E_{n-1}/E_n)}{\log(h_{n-1}/h_n)}.
\end{align*}
We use the following exact solutions:
\begin{align*}
\bu_{\text{ex}}(x,y) &= 
\begin{bmatrix}
  \cos(\pi y) \\
\sin(\pi x)
\end{bmatrix},\\
p_{F,\text{ex}}(x,y) &= \sin(\pi x)\sin(\pi y),\\
p_{T,\text{ex}}(x,y) &= p_{F,\text{ex}}(x,y) - \lambda \nabla \cdot \bu_{\text{ex}}(x,y),
\end{align*}
and define the terms on the right hand side, \(f\) and \(g\), correspondingly. We vary the parameters \( \lambda = \{1,10^8\}, K = \{1,10^{-8}\} \), for each refinement series, while letting \(\mu = 1\). One solution is depicted in Figure~\ref{fig:solution}. The convergence results are shown in Figure~\ref{fig:combined-convergence}. We observe optimal convergence rates, also when \(\lambda\) is large, and \(K\) is small. Note that the discrete norm for \(p_{F,h}\) \eqref{eq:normstar_3} is scaled by \(K\) and \(1/\lambda\).

\subsection{Geometrical robustness}
Next, we test the robustness of the method with respect to different cut configurations. We repeat the setup from the convergence analysis, and fix \(N=60\) and set \(\lambda = K = \mu = 1\). By slightly moving \(\widetilde{\Omega}\), we generate a family of translated background domains \(\{\widetilde{\Omega}_{\delta_k}\}_{k=1}^{2000}\) such that \(\widetilde{\Omega}_{\delta_k} = [-1+\delta_k h,1+\delta_k h ]^2 \) with \(\delta_k = k\cdot 5\cdot 10^{-4}\). We calculate the errors for each different cut configuration, both with and without stabilization terms, and report them in Figure~\ref{fig:robustness}. We see that the for the non-stabilized formulations, the error in the discrete norms \eqref{eq:normstar_1}-\eqref{eq:normstar_3} are very sensitive to the cut configuration, while the stabilized formulation is robust with respect to the cut configuration. For the error in the \(L^2\)-norm, there is also some sensitivity with the non-stabilized formulation.

\begin{figure}[t]
  \centering
  \includegraphics[ width=0.325\textwidth]{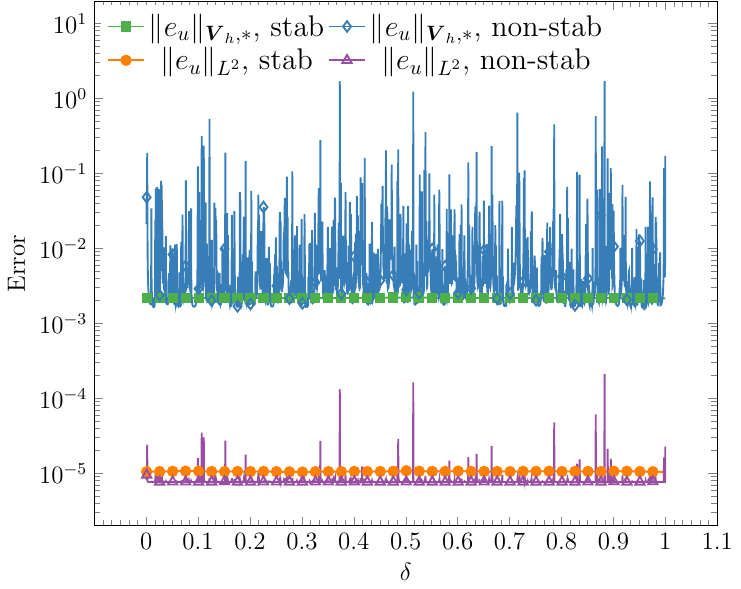}
  \includegraphics[ width=0.325\textwidth]{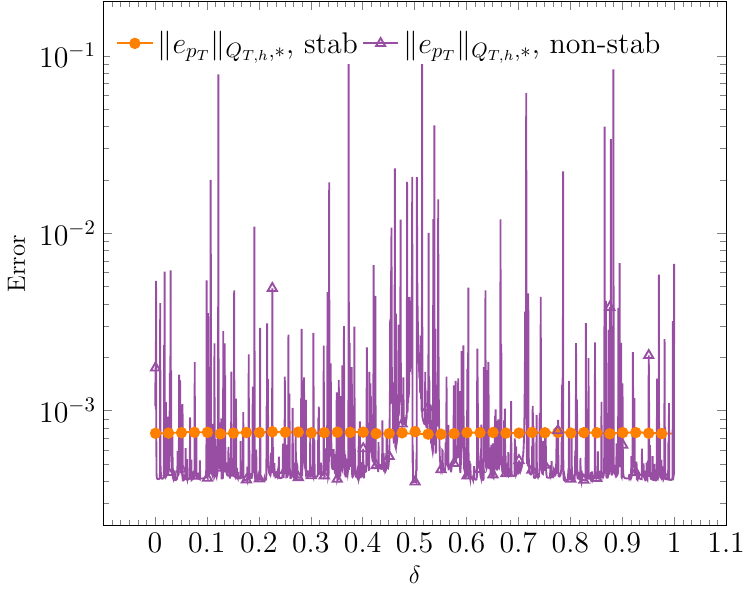}
  \includegraphics[ width=0.325\textwidth]{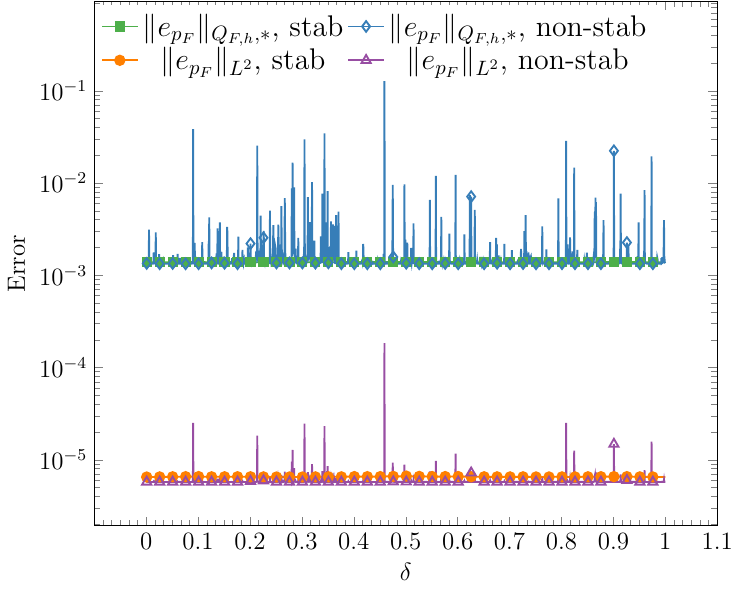}
  \caption{Geometrical robustness. Errors in both the discrete norms \eqref{eq:normstar_1}-\eqref{eq:normstar_3} and the \(L^2\)-norm for different cut configurations with and without stabilization.}
  \label{fig:robustness}
\end{figure}

\subsection{Convergence analysis 3D}
We next test the method in 3D, this time representing the physical domain by an STL mesh. We define the background domain as \(\widetilde{\Omega} = [-1.3,1.3]^3\). The STL mesh consists of a popcorn-like domain with a spherical hole in it, see Figure \ref{fig:solutions_3D}. The exact solutions is defined by:
\begin{align*}
  \bu_{\text{ex}}(x,y,z) &= 
\begin{bmatrix}
  \sin(\pi x) \cos(\pi y) \\
  \sin(\pi z) \\
  \cos(\pi x) \sin(\pi y) \sin(\pi z)
\end{bmatrix},\\ 
p_{F,\text{ex}}(x,y,z) &= \cos(\pi x)\sin(\pi x)\sin(\pi y),\\
p_{T,\text{ex}}(x,y,z) &= p_{F,\text{ex}}(x,y,z) - \lambda \nabla \cdot \bu_{\text{ex}}(x,y,z),
\end{align*}
We let \(\lambda = 5\), \(\mu=1\) and \(K=0.1\), and use a uniform mesh with \(N^3\) elements, where the element size \(h_n\) is gradually decreased by refining the mesh, setting \(N = [12,24,48]\). The convergence results are seen in Table~\ref{table:convergence3D}, and illustrations of the results in Figure~\ref{fig:solutions_3D}. We observe higher convergence than optimal convergence rates, which we conjecture is due to the geometrical resolution. 
 
\begin{table}
    \begin{center}
        \caption{Convergence analysis 3D. Errors and experimental orders of convergence (EOC).}
        \label{table:convergence3D}
        \begin{tabular}{
            l
            *{1}{S[table-format=1.2e-2]} *{1}{S[table-format=-1.2]}
            *{1}{S[table-format=1.2e-2]} *{1}{S[table-format=-1.2]}
            *{1}{S[table-format=1.2e-2]} *{1}{S[table-format=-1.2]}
        }
            \toprule
            {$N$}
            & {$\normvast{u-u_h} $}  & {EOC}
            & {$\norm{u-u_h}_{L^2} $} & {EOC}
            & {$\normtast{p_T-p_{T,h}} $} & {EOC} \\
            \midrule
            12 & 1.57e+00 & {--} & 1.62e-01 & {--} & 1.15e+00 & {--} \\
            24 & 1.89e-01 & 3.05 & 8.56e-03 & 4.25 & 2.30e-01 & 2.32 \\
            48 & 4.21e-02 & 2.17 & 6.85e-04 & 3.64 & 4.89e-02 & 2.23 \\
            \midrule
            {$N$}
            & {$\normfast{p_F-p_{F,h}}$} & {EOC}
            & {$\norm{p_F-p_{F,h}}_{L^2}$} & {EOC} \\
            \midrule
            12 & 6.03e-02 & {--} & 2.41e-02 & {--} \\
            24 & 7.38e-03 & 3.03 & 1.13e-03 & 4.42 \\
            48 & 1.61e-03 & 2.19 & 6.11e-05 & 4.21 \\
            \bottomrule
        \end{tabular}
    \end{center}
\end{table}

\begin{figure}[t]
  \centering
  \includegraphics[trim = {6cm 28cm 7cm 13cm},clip=true ,width=0.32\textwidth]{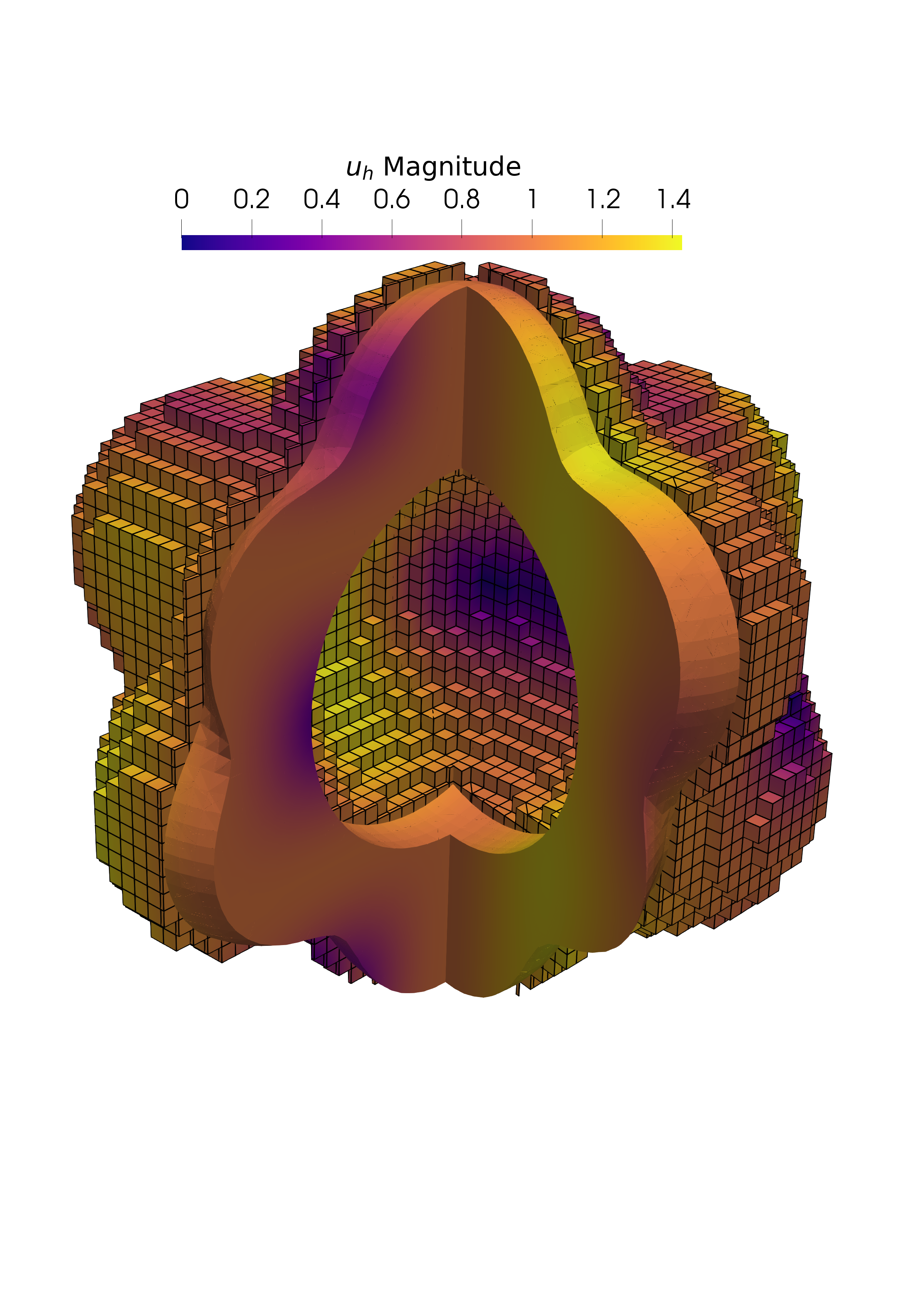}
  \includegraphics[trim = {6cm 28cm 7cm 13cm},clip=true, width=0.32\textwidth]{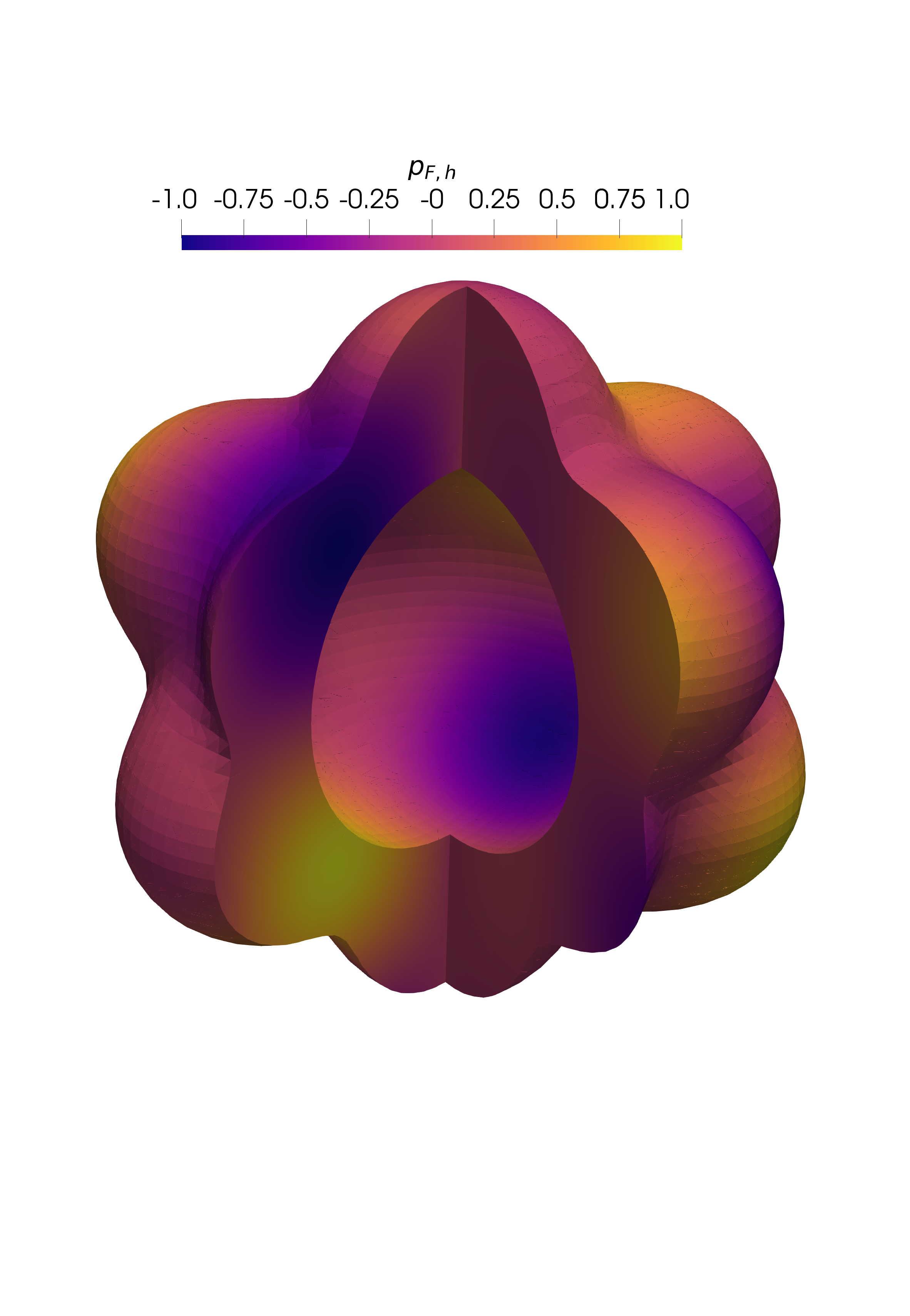}
  \includegraphics[trim = {6cm 28cm 7cm 13cm} ,clip=true,width=0.32\textwidth]{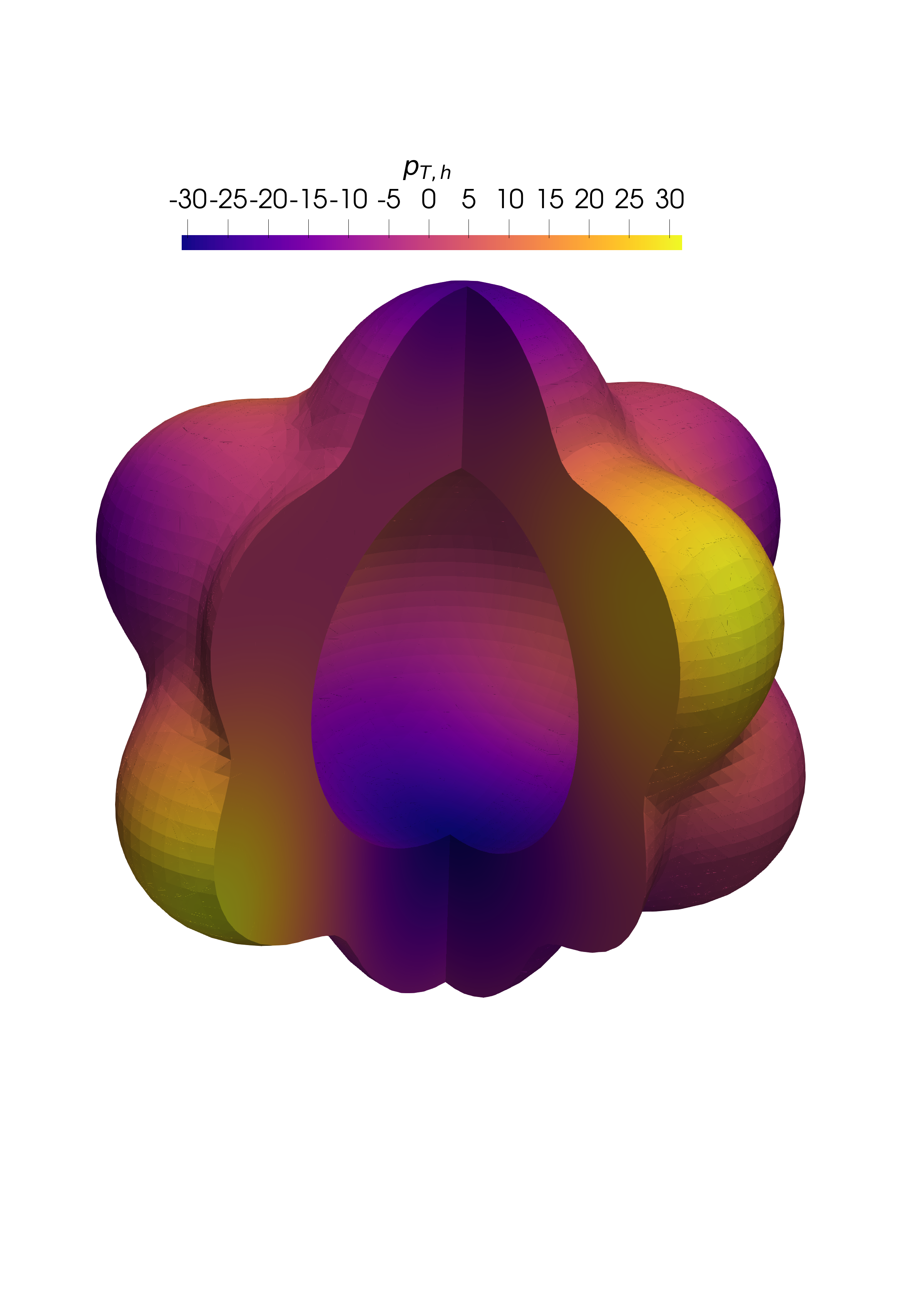}
  \caption{Convergence analysis 3D. The computational domain and the displacement solution \(\bu_h\) (magnitude), the fluid pressure solution \(p_{F,h}\), and  the  total pressure solution \(p_{T,h}\) for \(N=48\), here with a section removed for illustration.}
  \label{fig:solutions_3D}
\end{figure}

\subsection{Poroelastic simulation of brain mechanics}

In the final example we consider a poroelastic model of the brain for 
the interstitial fluid flow and tissue deformation in a setting with mixed boundary conditions.  Interstitial fluid flow caused by poroelastic deformations is important in the context of the glymphatic system~\cite{iliff2012paravascular}, see ~\cite{bohr2022glymphatic} for a recent review. In particular, an issue that has been debated is whether there is any static pressure driven advection acting on long time scales \cite{holter2017interstitial}.  In the present context, we therefore put focus on static forces represented by mixed boundary conditions. Here, we consider a setting with the brain mostly floating in cerebrospinal fluid (CSF) while parts are attached to dura matter, which is hard matter surrounding and protecting the brain. Specifically, it is known that while the intracranial pressure of the CSF pulsates with cardiac pulsation as well as the respiration with a magnitude of 100--1000 Pa, the pulsation is largely synchronous throughout the whole brain~\cite{vinje2019respiratory}. 
Therefore we ignore the pulsation and instead consider  
a small spatial gradient of ~1 Pa/cm~\cite{daversin2020mechanisms, stoverud2016poro}. Furthermore, 
as highlighted in studies of traumatic head injuries,  dura mater attachment prevents local motion and henceforth we consider dura attachment along the tentorium, see Figure~\ref{fig:brainsolutions_3D}, using 
a zero boundary condition for the displacement along with null flux. Elsewhere we assign a vertical pressure gradient, see e.g. ~\cite{croci2019uncertainty, rajna20193d} for more detailed pressure characterizations. Given such a setup, it is interesting to quantify the flow, pressure and displacement. Specifically, we solve system \eqref{eq:biot-total-press-I}--\eqref{eq:biot-total-press-III} on a right brain hemisphere represented by surface STL mesh available at \cite{martin_hornkjol_2024_10808334}. In the background mesh we set a mesh size of \(h=\SI{2.8125}{\milli\metre}\), giving the resulting system 1 million degrees of freedom. As boundary conditions we set 
\begin{align*}
\bu &= 0  & &\quad \mbox{ on tentorium} , 
\\
K \frac{\partial p_F}{\partial n} &= 0 & & \quad \mbox{ on tentorium} ,
\\
(\mu \varepsilon (\bu) - p_T \bI) \cdot \bn &= -p_{F,D}\bn  & & \quad \mbox{ on cerebrum} , 
\\
p_F & = p_{F,D}(x,y,z) = 1.33 - (z_0 - z)  \num {1.33e-3}  & &\quad \mbox{ on cerebrum},
\end{align*}
where \(z_0\) is the minimum z-coordinate of the domain. The parameters are set to \(E = \SI{1600}{\pascal}\), \(K = \SI{2.0e-6}{\milli\metre\squared\per\pascal}\), 
see \cite{stoverud2016poro}, the Lam\'e parameters are calculated as 
\begin{align*}
\mu = \frac{E}{2(1+\nu)}, \quad \lambda =  \frac{E \nu}{(1+\nu)(1-2\nu)},
\end{align*}
with \(\nu = \num{0.499}\),
and all source terms are set to zero. 
For the Nitsche parameter for the fluid pressure we have chosen the value higher than usual (\(\gamma_p = \num{4.0e6} \)) to better enforce the boundary condition on this relatively coarse mesh, as the dynamics are mainly driven by the gradient of the boundary pressure. The simulation results are presented in Figure~\ref{fig:brainsolutions_3D}. The displacement is of the order 0.1 mm, which is similar to that reported in the literature and measured through magnetic resonance imaging~\cite{goa2025brain}. The displacement field is however much more complex in real life, it varies a lot between individuals, and it is affected by multiple causes, such as cardiac, respiratory, sleep, etc~\cite{ulv2025sleep}. Concerning the interstitial fluid velocities, several studies have attempted to assess them in particular because of their importance for the glymphatic system. For example, \cite{holter2017interstitial} estimated pore velocities of up to 100 nm/s computationally, whereas \cite{ray2019analysis} estimated a superficial velocity of $\mu$m/s. Furthermore, in a parameter estimating study \cite{vinje2023human}, superficial velocities in the order of 2-4 $\mu$/min were found to provide best fit for overnight MRI-tracer movement. Our estimates here are at the lower end of the spectrum, with 
maximal velocities in the order of nm/s. We do remark that the static pressure gradient employed in our setting is related to the intracranial pressure pulsation occurring due to cardiac pulsation or respiration.  

\begin{figure}[t]
  \centering
  \includegraphics[width=0.45\textwidth]{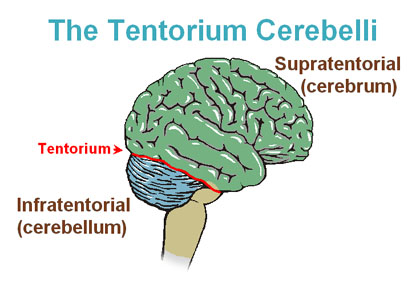}
  \includegraphics[ width=0.45\textwidth]{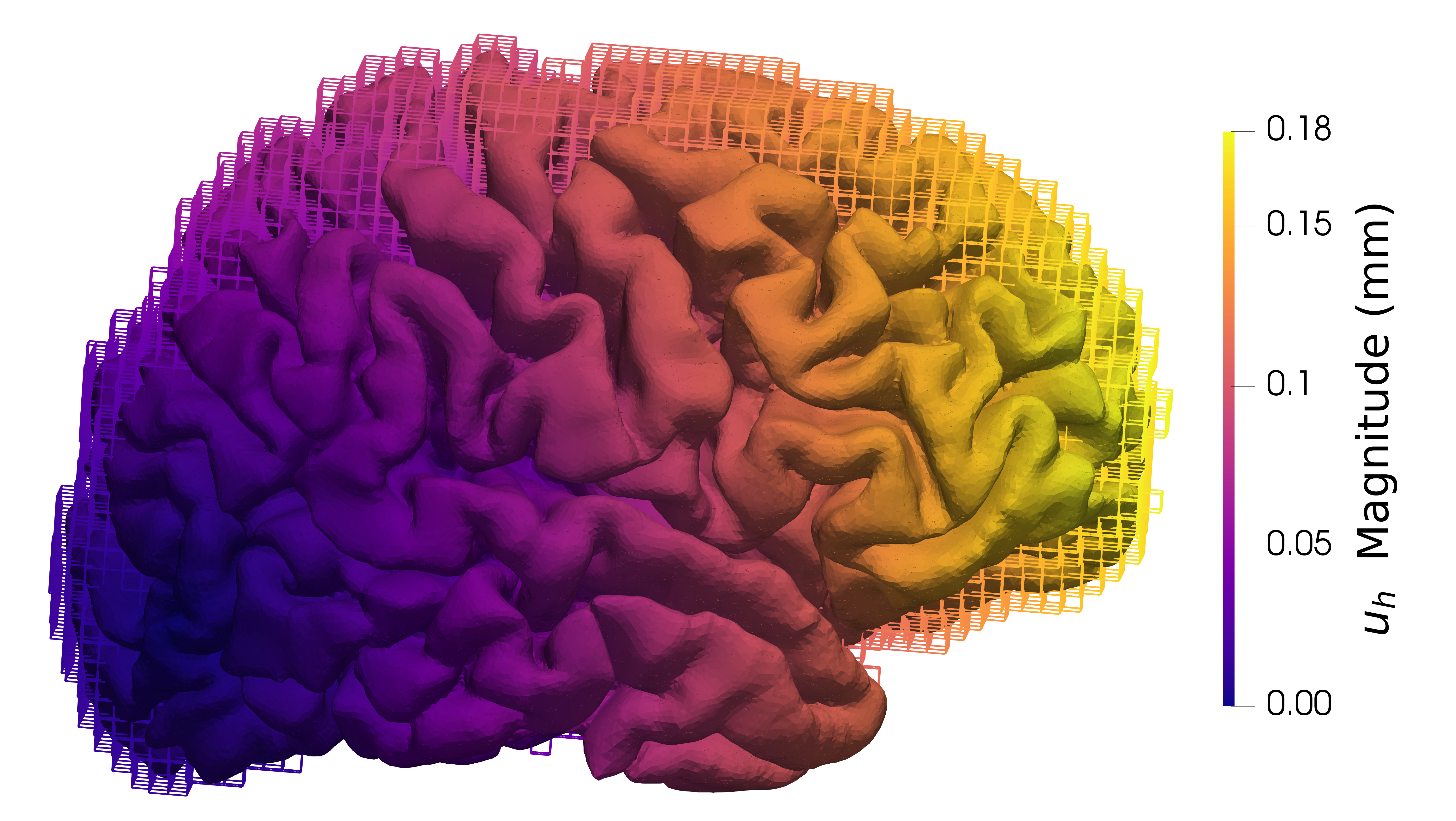}
\includegraphics[width=0.45\textwidth]{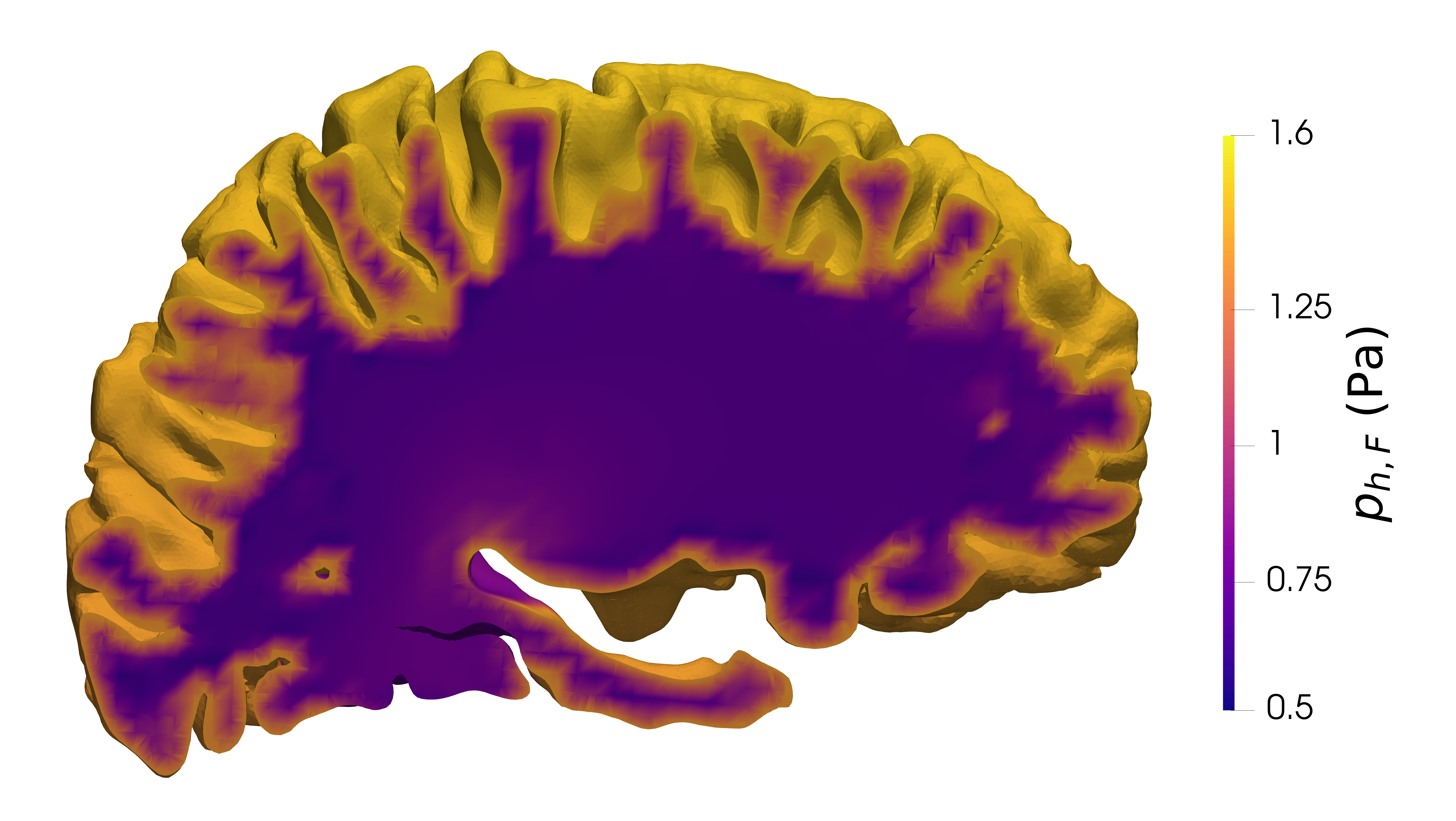}
\includegraphics[width=0.45\textwidth]{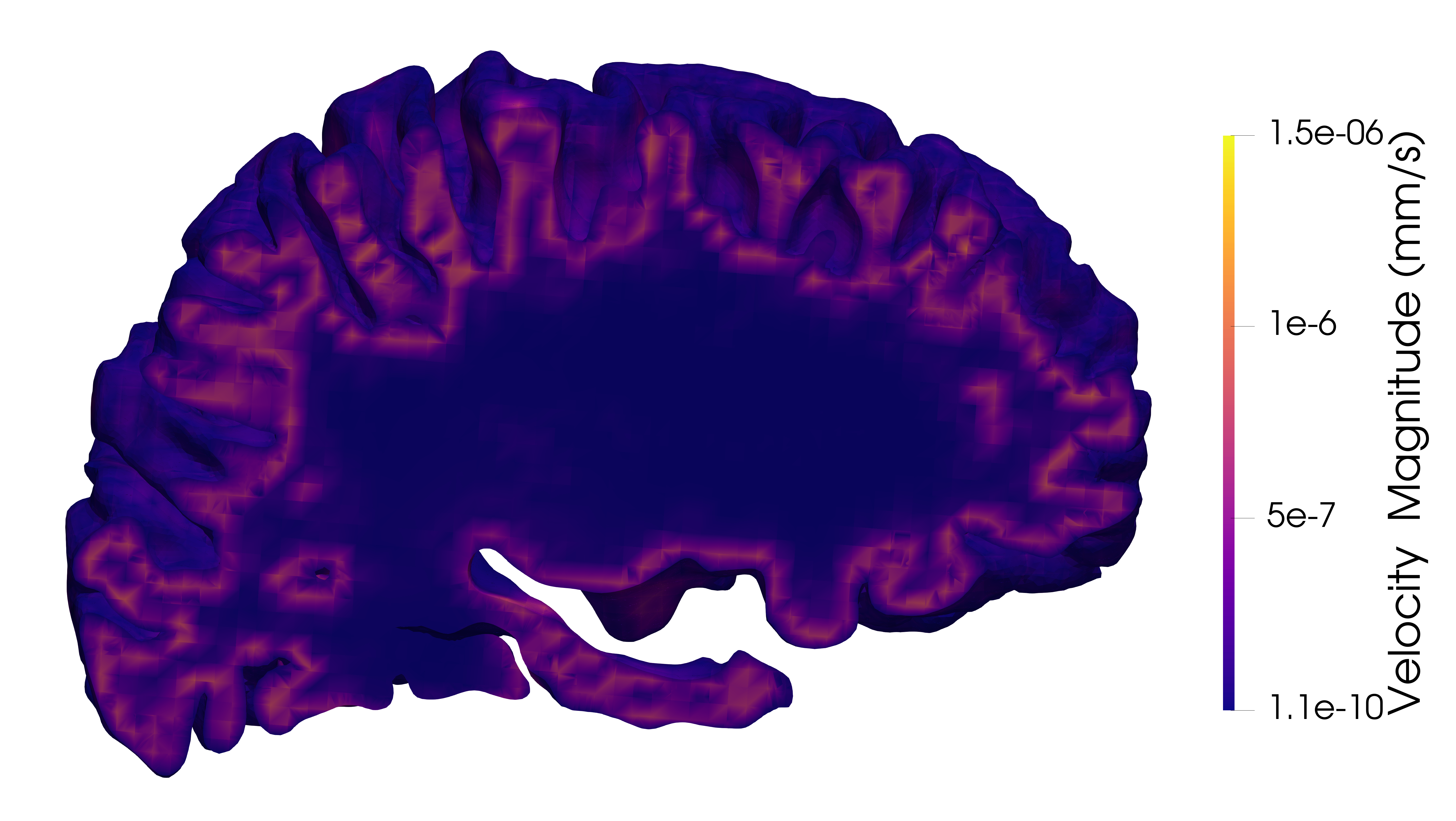}
  \caption{An illustration \protect\footnotemark (left, top) of the cerebellum, cerebrum and the tentorium, which separates these different parts of the central nervous system. The tentorium is a part of the dura matter and we simulate the case of attachment to the tentorium via zero displacement boundary condition in that region. Otherwise the brain is subjected to a stress field in the z-direction. The corresponding displacement (right, top)  is shown to be of a magnitude of almost 0.1 mm. The pressure gradients are localized near the boundary (left, bottom) as are the highest fluid velocities (right, bottom).
The displacement field is shown together with the (clipped) background mesh.}
  \label{fig:brainsolutions_3D}
\end{figure}
\footnotetext{\url{https://en.wikipedia.org/wiki/Cerebellar_tentorium}, license: public domain}


\section{Conclusion and outlook}
\label{sec:conclusion}
We have presented a stabilized cut finite element method for the Biot system based on the total-pressure formulation. Our theoretical analysis established stability and optimal convergence rates, independent of physical parameters and cut configurations, and these properties were confirmed by extensive numerical experiments in two and three dimensions. The CutFEM approach enables robust simulation of poroelasticity in complex domains, such as brain mechanics, without the need for body-fitted meshes.

Efficient preconditioners are essential for solving the Biot system at large scale, and future work should include the development and analysis of preconditioners for the proposed CutFEM formulation, e.g., based on~\cite{Lee-Mardal-Winther}. Additional directions include octree-based parallel adaptive mesh refinement~\cite{burstedde2011P4estScalableAlgorithms},
for which Gridap already provides an extension module. This would allow for local refinement along the embedded boundary and better resolution of boundary dynamics. 
Moreover, the geometric and discretization algorithms need to be extended to account for multiple interfaces and heterogeneous brain regions in more realistic multiphysics brain models with spatially varying material parameters. 
Other interesting multiphysics settings that could be considered are Biot-Stokes coupling~\cite{Stokes-Biot-eye,boon2022ParameterrobustMethodsBiotb} to include couplings with the surrounding CSF flow, or multiple-network poroelasticity (MPET) equations~\cite{lee2019MixedFiniteElement}.
A main advantage of the CutFEM approach in the context of biomedical applications is that meshing~\cite{mardal2022mathematical, dokken2025mathematical} is avoided, while still sub-voxel resolution of surfaces are retained through imaging segmentation tools such as for example FreeSurfer. 
\footnotetext{\url{https://en.wikipedia.org/wiki/Cerebellar_tentorium}, license: public domain}

\section*{Acknowledgments}
NB acknowledges support from the Swedish Research Council under grant no. 2021-06594 while in residence at Institut Mittag-Leffler in Djursholm, Sweden during the fall semester of 2025. 
KAM acknowledges funding by: Stiftelsen Kristian Gerhard Jebsen via the K.
G. Jebsen Centre for Brain Fluid Research; the national infrastructure for computational science in Norway Sigma2 via grant NN9279K; the Center of Advanced Study
at the Norwegian Academy of Science and Letters under the program Mathematical
Challenges in Brain Mechanics;  research stay funds at ICERM, Brown under the program
Numerical PDEs: Analysis, Algorithms, and Data Challenges; the Computational Hydrology project,  a strategic Sustainability initiative at the Faculty of 
Natural Sciences, UiO;  
and the grant 101141807 (aCleanBrain) by the European Research Council. 
AM acknowledges travel funds from: 
ICERM, Brown under the program "Numerical PDEs: Analysis, Algorithms, and Data Challenges";
Norwegian Academy of Science and Letters under the program "Mathematical Challenges in Brain Mechanics";
Institut Mittag-Leffler in Djursholm, Sweden, under the 
research program "Interfaces and unfitted discretization methods (grant no. 2021-06594 from the Swedish Research Council). IY is partially supported by the United States National Science Foundation grant DMS-2410686, the Alexander von Humboldt Foundation via the Humboldt Research Award, and the Center of Advanced Study
at the Norwegian Academy of Science and Letters under the program Mathematical
Challenges in Brain Mechanics.











\bibliographystyle{elsarticle-num}
\bibliography{references}
\end{document}